\documentclass[11pt,reqno]{amsart}

\usepackage{amsmath, amssymb, mathrsfs, amsthm}
\usepackage{bm}
\usepackage{float}
\usepackage{graphicx}
\usepackage{color}
\usepackage{array}
\usepackage{tabularx}
\usepackage{makecell}
\usepackage{epstopdf}
\usepackage{lineno}
\modulolinenumbers[5]
\usepackage{geometry}
\usepackage{booktabs}
\usepackage[colorlinks,linkcolor=blue,anchorcolor=blue,citecolor=blue,CJKbookmarks=True]{hyperref}
\usepackage[numbers,sort&compress]{natbib}

\usepackage[american]{babel}
\usepackage{microtype}
\usepackage{ragged2e}

\newtheorem{lemma}{Lemma}[section]
\newtheorem{theorem}{Theorem}[section]

\newtheorem{remark}{Remark}[section]
\newtheorem{proposition}{Proposition}[section]

\numberwithin{equation}{section}

\allowdisplaybreaks[4]

\begin{document}

\title[Strong error analysis of Euler methods for overdamped GLEs]{Strong error analysis of Euler methods for overdamped generalized Langevin equations with fractional noise:\ Nonlinear case}

\author{Xinjie Dai}
\address{Institute of Computational Mathematics and Scientific/Engineering Computing, Academy of Mathematics and Systems Science, Chinese Academy of Sciences, Beijing 100190, China}
\email{dxj@lsec.cc.ac.cn}

\author{Jialin Hong}
\address{Institute of Computational Mathematics and Scientific/Engineering Computing, Academy of Mathematics and Systems Science, Chinese Academy of Sciences, Beijing 100190, China; School of Mathematical Sciences, University of Chinese Academy of Sciences, Beijing 100049, China}
\email{hjl@lsec.cc.ac.cn}

\author{Derui Sheng}
\address{Institute of Computational Mathematics and Scientific/Engineering Computing, Academy of Mathematics and Systems Science, Chinese Academy of Sciences, Beijing 100190, China; School of Mathematical Sciences, University of Chinese Academy of Sciences, Beijing 100049, China}
\email{sdr@lsec.cc.ac.cn (Corresponding author)}

\author{Tau Zhou}
\address{Institute of Computational Mathematics and Scientific/Engineering Computing, Academy of Mathematics and Systems Science, Chinese Academy of Sciences, Beijing 100190, China; School of Mathematical Sciences, University of Chinese Academy of Sciences, Beijing 100049, China}
\email{zt@lsec.cc.ac.cn}

\thanks{This work is supported by National key R\&D Program of China under Grant No.\ 2020YFA0713701, and National Natural Science Foundation of China (Nos.\ 11971470, 11871068, 12031020, 12022118).}

\keywords{generalized Langevin equation, fractional Brownian motion, singular kernel, fast Euler method, multilevel Monte Carlo simulation, Malliavin calculus}

\begin{abstract}
This paper considers the strong error analysis of the Euler and fast Euler methods for nonlinear overdamped generalized Langevin equations driven by the fractional noise. The main difficulty lies in handling the interaction between the fractional Brownian motion and the singular kernel, which is overcome by means of the Malliavin calculus and fine estimates of several multiple singular integrals. Consequently, these two methods are proved to be strongly convergent with order nearly $\min\{2(H+\alpha-1), \alpha\}$, where $H \in (1/2,1)$ and $\alpha\in(1-H,1)$ respectively characterize the singularity levels of fractional noises and singular kernels in the underlying equation. This result improves the existing convergence order $H+\alpha-1$ of Euler methods for the nonlinear case, and gives a positive answer to the open problem raised in \cite{FangLi2020}. As an application of the theoretical findings, we further investigate the complexity of the multilevel Monte Carlo simulation based on the fast Euler method, which turns out to behave better performance than the standard Monte Carlo simulation when computing the expectation of functionals of the considered equation. 
\end{abstract}

\maketitle

\textit{AMS subject classifications}: 65C20, 65C30, 65C05, 60H07

\section{Introduction} 

The \textit{generalized Langevin equation} (GLE) was originally introduced by \cite{Mori1965} and later used extensively to describe the subdiffusion within a single protein molecule \cite{Kou2008, KouSunneyXie2004}, the motion of microparticles moving randomly in viscoelastic fluids \cite{DidierNguyen2020, McKinleyNguyen2018}, and so on. To be specific, the position $x (t)$ of a moving particle with mass $m$ in the energy potential $V$ at time $t$ can be modelled by the GLE
\begin{align*}
m \ddot x(t) = - \nabla V(x(t)) - \int_0^t K(t-s) \dot x(s) \mathrm d s + \eta(t).
\end{align*}
Here, dot denotes the derivative on time, and the convolutional kernel $K(t)$ of the friction (dissipation) is related to the random force (fluctuation) $\eta(t)$ through the \textit{fluctuation--dissipation theorem} (FDT)
\begin{align*}
\mathbb E \big[ \eta(t) \eta(s) \big] = k_{B} T_A K(t-s), \quad \mbox{for } s \leq t,
\end{align*}
where $k_{B}$ is Boltzmann's constant and $T_A$ is the absolute temperature (see e.g., \cite{Kubo1966}). To capture the ubiquitous memory phenomena in biology and physics, the fluctuation $\eta(t)$ is often characterized by the fractional noise, and then the FDT reveals the memory kernel $K(t)$ being proportional to a power law $t^{-\alpha}$ with some $\alpha > 0$ (see e.g., \cite{Kou2008, KouSunneyXie2004}). In the `overdamped' regime ($m\ll 1$), the GLE with fractional noise reduces to the following fractional \textit{stochastic differential equation} (SDE)
\begin{align} \label{eq.CaputoFSDE}
D_c^{\alpha} x(t) = b(x(t)) + \sigma \dot W_H(t),
\end{align}
where $b := -\nabla V$, $D_c^{\alpha} x(t) := \frac{1}{\Gamma (1-\alpha)} \int_0^t (t-s)^{-\alpha} \dot x (s) \mathrm d s$ is the Caputo fractional derivative with $\Gamma$ being the Gamma function, $\sigma > 0$ is the noise intensity, and $W_H$ denotes the \textit{fractional Brownian motion} (fBm) with Hurst index $H\in(1/2,1)$. Eq.\ \eqref{eq.CaputoFSDE}, also known as the overdamped GLE with fractional noise, shall be mathematically interpreted by its integral form (see \eqref{eq.GLE} for details). It is a class of \textit{stochastic Volterra integral equations} (SVIEs), and we refer to \cite{LiLiu2019, LiLiu2017} for more theoretical results on the well-posedness and long-time behavior of the exact solution.

This paper is concerned with discrete-time simulations of Eq.\ \eqref{eq.CaputoFSDE}, in view of the absence of closed-form solutions. Generally, the memory kernels will result in expensive costs when performing standard time discretizations such as the Euler method. As an appropriate candidate, the fast Euler method shares a satisfactory computational efficiency, which is constructed by combining the Euler method with the sum-of-exponentials approximation. The prerequisite of the strong error analysis of the fast Euler method is to estimate the strong error of the Euler method. Following the arguments in \cite[Proposition 3.1]{FangLi2020}, the strong convergence order $H+\alpha-1$ of the Euler method is available for Eq.\ \eqref{eq.CaputoFSDE} with $\alpha\in(1-H,1)$, which exactly coincides with the mean square H\"{o}lder continuity exponent of the exact solution. Clearly, one can expect a higher convergence order in terms of Eq.\ \eqref{eq.CaputoFSDE} since the driven noise is additive. For the harmonic potential case, which corresponds to the linear external force case, it is firstly proved in \cite{FangLi2020} that the Euler method applied to Eq.\ \eqref{eq.CaputoFSDE} with $\alpha=2-2H$ is strongly convergent with the sharp order nearly $\min\{3-3H,3/2-H\}$. This convergence order result was later extended by \cite{DaiXiao2021} for general $\alpha\in(1-H,1)$. For the nonlinear external force case, the authors of \cite[Page 440]{FangLi2020} left the improvement of the strong convergence order of the Euler method as an open problem, which is exactly one main goal of the present paper.

The key-point of the strong convergence analysis for the Euler method consists in the upper bound estimate of
\begin{align} \label{introQuantity}
\mathbb{E}\left[ b'\big( \xi_{s}^{\theta} \big) b'\big( \xi_{\tau}^{\lambda} \big) \big(G(s) - G(\hat{s})\big) \big(G(\tau) - G(\hat{\tau})\big) \right],
\end{align}
where $\xi_{s}^\theta := (1-\theta) x(\hat{s}) + \theta x(s)$ with $\hat{s}$ denoting the maximal grid point before $s$, and $G(\cdot) := \frac{\sigma}{\Gamma(\alpha)} \int_0^\cdot (\cdot-s)^{\alpha-1} \mathrm{d}W_H(s)$ is a singular stochastic integral; see Proposition \ref{lem.covGt} for more related notations. In contrast with the cases of SDEs with fBm and SVIEs with standard Brownian motion, the treatment of \eqref{introQuantity} is more difficult due to the interaction between the fBm and singular kernels, even for the linear external force case; see \cite[Lemma 1]{DaiXiao2021} for more details. For the nonlinear external force case, we adopt the dual formula in Malliavin calculus to convert stochastic integrals in \eqref{introQuantity} into deterministic ones, whose integrands involve the first and second order Malliavin derivatives of the exact solution. It turns out that the Malliavin derivatives of the exact solution are bounded by some quantities associated to the corresponding singular kernel, rather than by some constant in the case of SDEs (see e.g., \cite{KloedenNeuenkirch2011}).
Consequently, more complicated multiple singular integrals need to be handled in our case. By delicately estimating these multiple singular integrals, we attain the strong convergence order nearly $\min\{2(H+\alpha-1), \alpha\}$ of the Euler method for Eq.\ \eqref{eq.CaputoFSDE} with $\alpha\in(1-H,1)$. On this basis, we can also read that the strong error of the fast Euler method with tolerance $\epsilon$ is bounded by that of the Euler method plus multiples of $\epsilon$. The new strong convergence order nearly $\min\{2(H+\alpha-1), \alpha\}$ improves the existing convergence order $H+\alpha-1$ of Euler methods in the nonlinear case, and particularly reproduces the corresponding result of \cite[Theorem 1]{KloedenNeuenkirch2011} for SDEs with fBm.


An important application of strong error analysis of numerical methods is to analyze the complexity of the \textit{multilevel Monte Carlo} (MLMC) simulation that was originally developed by Giles \cite{Giles2008} to approximate the expectation of functionals of SDEs with standard Brownian motion. Compared with the standard Monte Carlo simulation, the MLMC simulation has better performance when computing such quantities, and has been applied successively to different equations with various noises (see \cite{KloedenNeuenkirch2011} for SDEs with fBm, \cite{RichardTan2021} for SVIEs with standard Brownian motion, and so on). By computing corrections using multiple levels of grids, the MLMC simulation reduces the variance of the estimator to achieve the fine grid accuracy at a relatively low cost, where the variance has a close relationship with the strong convergence order of the time discretization. Taking into account that the fast Euler method is more efficient than the Euler method, we apply the MLMC simulation based on the fast Euler method to compute the expectation of functionals of Eq.\ \eqref{eq.CaputoFSDE}. The corresponding complexity analysis is meanwhile investigated; see Theorem \ref{thm.MLMCComplexity} for more details.

The paper is organized as follows. Section \ref{sec.mainRes} presents the theoretical findings on the Euler method, fast Euler method and MLMC simulation when they are applied to Eq.\ \eqref{eq.CaputoFSDE}. In order to facilitate the proof of main results, we study the first and second order Malliavin derivatives of the exact solution, and establish the estimates of several multiple singular integrals in Section \ref{sec.Preli}. Section \ref{sec.Proof} provides the detailed proof of main results. 

\textbf{Notations.} Denote $a \vee b := \max\{a, b\}$ and $a \wedge b := \min\{a, b\}$ for $a,\, b \in \mathbb{R}$. For the integer $m \geq 1$, denote by $C_b^m$ the space of not necessarily bounded real-valued functions that have continuous and bounded derivatives up to order $m$, and by $C_p^m$ the space of $m$ times continuously differentiable real-valued functions whose derivatives up to order $m$ are of at most polynomial growth. For the integer $l \geq 2$, let $C_{b,p}^{1,l} := C_b^1 \cap C_p^l$. For any $q\geq1$, $\|\cdot\|_{q}$ denotes the $L^{q}(\Omega;\mathbb R)$-norm, and particularly $\|\cdot\| := \|\cdot\|_2$. Denote by $\langle \cdot, \cdot \rangle$ the $L^{2}(\Omega;\mathbb R)$-inner product. Let $\mathbf{1}_{S} (\cdot)$ be the indicator function of the set $S$. Use $C$ as a generic constant and use $C(\cdot)$ if necessary to mention the parameters it depends on, whose values are always independent of the stepsize $h$ and may change when it appears in different places.

\section{Main results} \label{sec.mainRes}

Throughout this paper, we restrict ourselves to the case of 1-dimension for the simplicity of notations, and remark that all results could be extended to the multi-dimension case.
Then the overdamped GLE \eqref{eq.CaputoFSDE} is mathematically interpreted by \begin{align}\label{eq.GLE}
x(t) = x_0 + \frac{1}{\Gamma(\alpha)}\int_0^{t} (t-s)^{\alpha-1} b(x(s)) \mathrm{d}s + \frac{\sigma}{\Gamma(\alpha)} \int_0^t (t-s)^{\alpha-1} \mathrm{d}W_H(s) 
\end{align}
for $t \in [0, T]$, where $W_H$ is a 1-dimensional fBm on some complete filtered probability space $(\Omega,\mathscr F,\{\mathscr F_t\}_{t\in[0,T]}, \mathbb P)$. We always assume that
$b:\mathbb R\rightarrow \mathbb R$ is a Lipschitz continuous function and the initial value $x_0\in\mathbb R$ is deterministic. In this setting, Eq.\ \eqref{eq.GLE} has a unique strong solution for $H\in(1/2,1)$ and $\alpha\in(1-H,1)$ (see \cite[Theorem 1]{LiLiu2017}).

In order to solve Eq.\ \eqref{eq.GLE} numerically, we introduce the Euler method and fast Euler method, where the latter is more computationally efficient. Meanwhile, the corresponding strong error analysis is established.

\subsection{Euler method and fast Euler method}

For a fixed integer $N \geq 2$, let $\{ t_n := n h, \, n= 0, 1, \cdots, N \}$ be a uniform partition of $[0,T]$ with the stepsize $h := T/N$. As introduced in \cite{FangLi2020}, the Euler method for Eq.\ \eqref{eq.GLE} can be formulated as
\begin{align}\label{eq.EM}
x_n = x_0 + \frac{1}{\Gamma(\alpha)} \sum_{j=1}^{n} b(x_{j-1}) \int_{t_{j-1}}^{t_j} (t_n-s)^{\alpha-1} \mathrm{d}s + G(t_n),
\end{align}
for $n = 1,2,\cdots,N$, in which
\begin{align}\label{eq.Gt}
G(t_n) := \frac{\sigma}{\Gamma(\alpha)} \int_0^{t_n} (t_n-s)^{\alpha-1} \mathrm{d}W_H(s).
\end{align}
Now we are in a position to present our first main result on the strong convergence order of the Euler method \eqref{eq.EM} for Eq.\ \eqref{eq.GLE} in Theorem \ref{thm.mainEM}, whose proof is deferred to Section \ref{sec.Proof}.

\begin{theorem} \label{thm.mainEM}
Let $H \in (1/2, 1)$ and $\alpha \in (1-H, 1)$. If $b \in C_{b,p}^{1,3}$, then there exists some positive constant $C$ such that the strong error of the Euler method \eqref{eq.EM} can be controlled as
\begin{align*} 
\sup\limits_{n \leq T/h} \| x_n - x(t_n) \| \leq C \mathcal{R}_{H, \alpha} (h),
\end{align*}
where
\begin{align*}
\mathcal{R}_{H, \alpha} (h) :=
\begin{cases}
h^{2(H+\alpha-1)}, \quad &\mbox{if } \alpha \in (1-H,2-2H), \\
(|\ln h| \vee \ln T) h^{2-2H}, \quad &\mbox{if } \alpha = 2-2H, \\
h^{\alpha}, \quad &\mbox{if } \alpha \in (2-2H,1).
\end{cases}
\end{align*}
\end{theorem}

\begin{remark}
For the case $\alpha=1$, the model \eqref{eq.GLE} reduces to the SDE
\begin{align}\label{eq.SDE}
\mathrm{d} x(t) = b(x(t)) \mathrm{d}t + \sigma \mathrm{d} W_H(t),\qquad t\in[0,T]
\end{align}
with $x(0)=x_0$. When $\alpha$ tends to $1$, Theorem \ref{thm.mainEM} reproduces Theorem 1 in \cite{KloedenNeuenkirch2011}, which presented that the Euler method for Eq.\ \eqref{eq.SDE} is of first-order strong convergence under a slightly stronger assumption $b \in C_b^3$. While for the case $\alpha\in(1-H,2-2H)$, Theorem \ref{thm.mainEM} indicates that the strong convergence order $2(H+\alpha-1)$ of the Euler method is twice of the mean square H\"older continuity exponent of the exact solution (see \cite[Lemma 2.3]{FangLi2020}). The new strong convergence order nearly $\min\{2(H+\alpha-1), \alpha\}$ improves the existing convergence order $H+\alpha-1$ of the Euler method in the nonlinear case (see \cite[Proposition 3.1]{FangLi2020}).
\end{remark}

Since the overdamped GLE \eqref{eq.GLE} is an SVIE with memory, the Euler method \eqref{eq.EM} needs the computational cost of $\mathcal{O}(N^2)$ for a single sample path, which is too expensive in practical calculations. To improve the computational efficiency, \cite{FangLi2020} proposed the fast Euler method by using the following sum-of-exponentials approximation.

\begin{lemma}[Sum-of-exponentials approximation \cite{FangLi2020, JiangZhang2017}]
For $\alpha\in(0,1)$, tolerance $\epsilon > 0$ and truncation $\kappa > 0$, there exist positive numbers $\tau_i$ and $\omega_i$ with $1 \leq i \leq M_{\exp}$ such that
\begin{align}\label{eq.SOE}
\Big| t^{\alpha - 1} - \sum_{i=1}^{M_{\exp}} \omega_i e^{-\tau_i t} \Big| \leq \epsilon, \qquad \forall\, t\in[\kappa,T],
\end{align}
where
\begin{align*}
M_{\exp} = \mathcal{O} \left(\log\frac{1}{\epsilon} \Big(\log\log\frac{1}{\epsilon} + \log\frac{T}{\kappa}\Big)
+ \log\frac{1}{\kappa}\Big(\log\log\frac{1}{\epsilon} + \log\frac{1}{\kappa}\Big)\right).
\end{align*}
\end{lemma}

Using the sum-of-exponentials approximation \eqref{eq.SOE} with a given tolerance $\epsilon \ll 1$, the Euler method \eqref{eq.EM} can be directly modified to
\begin{align*}
y_n &= y_0 + \sum_{j=1}^{n-1} \frac{b(y_{j-1})}{\Gamma(\alpha)} \int_{t_{j-1}}^{t_j} \sum_{i=1}^{M_{\exp}} \omega_i e^{-\tau_i (t_n-s)} \mathrm{d}s + \frac{b(y_{n-1})}{\Gamma(\alpha)} \int_{t_{n-1}}^{t_n} (t_n-s)^{\alpha-1} \mathrm{d}s + G(t_n).
\end{align*}
Then, exchanging the summations order obtains the fast Euler method
\begin{align}\label{eq.FEM}
y_n &= y_0 + \sum_{i=1}^{M_{\exp}} \omega_i \zeta_i^n + \frac{h^{\alpha}}{\Gamma(\alpha+1)} b(y_{n-1})+ G(t_n),
\end{align}
where
\begin{align*}
\zeta_i^n :=
\begin{cases}
 0, & \mbox{if } n=1, \\
 \frac{1}{\Gamma(\alpha)} \sum_{j=1}^{n-1} b(y_{j-1}) \int_{t_{j-1}}^{t_j} e^{-\tau_i (t_n-s)} \mathrm{d}s, & \mbox{if } n \geq 2
\end{cases}
\end{align*}
satisfies the recurrence formula
\begin{align*}
\zeta_i^{n+1} = e^{-\tau_i h} \zeta_i^n + \frac{1}{\tau_i \Gamma(\alpha)} ( e^{-\tau_i h} - e^{-2\tau_i h} ) b(y_{n-1}).
\end{align*}

If one takes $\epsilon = h^{\alpha}$ and $\kappa = h$, then $M_{\exp} = \mathcal{O} ( (\log N)^2 )$, and the above recurrence formula indicates that the fast Euler method \eqref{eq.FEM} only needs a computational cost of $\mathcal{O} ( N(\log N)^2 )$ for a single sample path, so it is more efficient than the Euler method \eqref{eq.EM}; see also \cite{FangLi2020, JiangZhang2017} for more details. For the fast Euler method \eqref{eq.FEM}, we provide the following strong convergence theorem.

\begin{theorem} \label{thm.mainFEM}
Under the assumptions of Theorem \ref{thm.mainEM}, there exists a positive constant $C$ independent of $0 < \epsilon \ll 1$ such that the strong error of the fast Euler method \eqref{eq.FEM} can be bounded by
\begin{align*}
\sup\limits_{n \leq T/h} \| y_n - x(t_n) \| \leq C \mathcal{R}_{H, \alpha} (h) + C \epsilon,
\end{align*}
where $\mathcal{R}_{H, \alpha} (h)$ is defined in Theorem \ref{thm.mainEM}.
\end{theorem}
\begin{proof}
Based on the result of Theorem \ref{thm.mainEM}, the proof can be completed similar to that of \cite[Theorem 4.4]{FangLi2020}. Thus, the details are omitted.
\end{proof}

\subsection{Multilevel Monte Carlo simulation}

We are also interested in the calculation of $\mathbb{E} [ f(x(T)) ]$ for some Lipschitz continuous function $f:\ \mathbb{R} \rightarrow \mathbb{R}$, because this quantity receives a lot of attention in applications (see e.g., \cite{Giles2015}). With Theorem \ref{thm.mainFEM} in hand, one can construct the MLMC simulation combined with the fast Euler method \eqref{eq.FEM} to compute $\mathbb{E} [ f(x(T)) ]$, where the MLMC simulation was originally proposed by Giles \cite{Giles2008} to improve the efficiency of Monte Carlo simulations.

Without loss of generality, we assume that $T = 1$ in this subsection. For a fixed integer $M \geq 2$, define different stepsizes $h_l = M^{-l} \ (l = 0, 1, \cdots, L)$. For convenience, denote
\begin{align} \label{eq.QPl}
Q = f(x(1)), \quad\quad P_l = f(y_{M^l}^{h_l}), \quad l = 0, 1, \cdots, L,
\end{align}
where $y_{M^l}^{h_l}$ is the approximation of $x(1)$ by using the fast Euler method \eqref{eq.FEM} with stepsize $h_l$ and tolerance $h_l^{\alpha}$. Based on the trivial identity
\begin{align*}
\mathbb{E} [P_L] = \mathbb{E} [P_0] + \sum_{l=1}^{L} \mathbb{E} [P_l - P_{l-1}],
\end{align*}
the MLMC estimator can be formulated as
\begin{align}\label{eq.MLMCFEM}
Z := \frac{1}{N_0} \sum_{n=1}^{N_0} P_0^{(n)} + \sum_{l=1}^{L} \frac{1}{N_l} \sum_{n=1}^{N_l} [P_l^{(n)} - P_{l-1}^{(n)}],
\end{align}
where $\frac{1}{N_0} \sum_{n=1}^{N_0} P_0^{(n)}$ with $N_0$ i.i.d.\ copies $\{P_0^{(n)}\}_{n=1}^{N_0}$ of $P_0$ is used to estimate $\mathbb{E} [P_0]$,\\$\frac{1}{N_l} \sum_{n=1}^{N_l} [P_l^{(n)} - P_{l-1}^{(n)}]$ with $N_l$ i.i.d.\ copies $\{P_l^{(n)}, P_{l-1}^{(n)}\}_{n=1}^{N_l}$ of $\{P_l, P_{l-1}\}$ is used to estimate $\mathbb{E} [P_l - P_{l-1}]$ for $l = 1,2,\cdots, L$, and the positive integers $L$ and $N_l\ (l = 0, 1, \cdots, L)$ are related to the following complexity theorem.

\begin{theorem} \label{thm.MLMCComplexity}
Let $H \in (1/2,1)$ and $\alpha = 2-2H$. Under the assumptions of Theorem \ref{thm.mainEM}, there exist suitable positive integers $L$ and $N_l\ (l = 0, 1, \cdots, L)$ such that the mean square error of the MLMC estimator \eqref{eq.MLMCFEM} can be controlled by a specified accuracy $\varepsilon$~$(0 < \varepsilon \ll 1)$, i.e.,
\begin{align*}
\| Z - \mathbb{E} [ f(x(1)) ] \| \leq \varepsilon
\end{align*}
with the computational cost $C_{\mbox{\tiny MLMC}}$ satisfying
\begin{align*}
C_{\mbox{\tiny MLMC}} \leq
\begin{cases}
C \varepsilon^{-2}, & \mbox{if } H \in (\frac{1}{2}, \frac{3}{4}), \\
C \varepsilon^{-2} |\ln \varepsilon|^6, & \mbox{if } H = \frac{3}{4}, \\
C \varepsilon^{-(2+\frac{4H-3}{2-2H-\rho})} |\ln \varepsilon|^4, & \mbox{if } H \in (\frac{3}{4}, 1),
\end{cases}
\end{align*}
where $\rho \in (0,1-H)$ can be arbitrary small, and the positive constant $C$ is independent of $\varepsilon$.
\end{theorem}

\begin{proof}
The variance formula $\mbox{Var}(X) = \| X \|^2 - |\mathbb E[X]|^2$, for $X \in L^2(\Omega;\mathbb{R})$, yields
\begin{equation} \label{eq.RMS}
\begin{split}
&\quad\ \| Z - \mathbb{E} [ f(x(1)) ] \|^2 = \big| \mathbb{E} \big[ Z - \mathbb{E} [ f(x(1)) ] \big] \big|^2 + \mbox{Var}(Z) \\
&= \big| \mathbb{E} \big[ P_L - Q \big] \big|^2 + \frac{1}{N_0} \mbox{Var}(P_0) + \sum_{l=1}^{L} \frac{1}{N_l} \mbox{Var}(P_l - P_{l-1}) =: \uppercase\expandafter{\romannumeral1} + \uppercase\expandafter{\romannumeral2} + \uppercase\expandafter{\romannumeral3},
\end{split}
\end{equation}
where $Q$ and $P_l$ are defined by \eqref{eq.QPl}. Firstly, Theorem \ref{thm.mainFEM} and the Lipschitz continuity of $f$ show that
\begin{align} \label{eq.MLMCPart1}
&\quad\ \uppercase\expandafter{\romannumeral1} := \big| \mathbb{E} \big[ P_L - Q \big] \big|^2 \leq \| P_L - Q \|^2 = \| f(y_{M^L}^{h_L}) - f(x(1)) \|^2 \nonumber \\
&\leq C \| y_{M^L}^{h_L} - x(1) \|^2 \leq C h_{L}^{4-4H} |\ln h_{L}|^2 \leq C_1 h_{L}^{4-4H-2\rho} \leq \frac{\varepsilon^2}{3}
\end{align}
provided $\rho \in (0,1-H)$ and
\begin{align} \label{eq.MLMClevel}
L = \left\lceil \frac{1}{4-4H-2\rho} \log_M \big( 3 C_1 \varepsilon^{-2} \big) \right\rceil,
\end{align}
where the notation $\lceil \cdot \rceil$ denotes the ceiling function. Secondly,
\begin{align} \label{eq.MLMCPart2}
\uppercase\expandafter{\romannumeral2} := \frac{1}{N_0} \mbox{Var}(P_0) \leq \frac{\| P_0 \|^2}{N_0} \leq \frac{C_2}{N_0} \leq \frac{\varepsilon^2}{3}
\end{align}
provided
\begin{align} \label{eq.MLMCN0}
N_0 = \left\lceil 3 C_2 \varepsilon^{-2} \right\rceil.
\end{align}
Thirdly, for $l = 1, 2, \cdots, L$, similar to the estimation of \eqref{eq.MLMCPart1}, one can obtain
\begin{align*}
\mbox{Var}(P_l - Q) \leq \| P_l - Q \|^2 \leq C h_l^{4-4H} |\ln h_l|^2
\end{align*}
and
\begin{align*}
\mbox{Var}(Q - P_{l-1}) \leq C h_{l}^{4-4H} |\ln h_{l}|^2,
\end{align*}
so
\begin{align*}
&\quad\ \mbox{Var}(P_l - P_{l-1}) = \mbox{Var}(P_l - Q + Q - P_{l-1}) \\
&\leq \Big( \sqrt{\mbox{Var}(P_l - Q)} + \sqrt{\mbox{Var}(Q - P_{l-1})} \Big)^2 \leq C_3 h_{l}^{4-4H} |\ln h_{l}|^2.
\end{align*}
Thus,
\begin{align} \label{eq.MLMCPart3}
\uppercase\expandafter{\romannumeral3} := \sum_{l=1}^{L} \frac{1}{N_l} \mbox{Var}(P_l - P_{l-1}) \leq \sum_{l=1}^{L} \frac{C_3}{N_l} h_{l}^{4-4H} |\ln h_{l}|^2 \leq \frac{\varepsilon^2}{3}
\end{align}
provided
\begin{align} \label{eq.MLMCNl}
N_l = \left\lceil 3 C_3 \varepsilon^{-2} h_{l}^{\frac{5-4H}{2}} |\ln h_{l}|^2 \sum_{l'=1}^{L} h_{l'}^{\frac{3-4H}{2}} \right\rceil, \quad \mbox{for } l = 1, 2, \cdots, L.
\end{align}
Combining \eqref{eq.RMS}, \eqref{eq.MLMCPart1}, \eqref{eq.MLMCPart2} and \eqref{eq.MLMCPart3}, one can take $L$ and $N_l \ (l = 0, 1, \cdots, L)$ satisfying \eqref{eq.MLMClevel}, \eqref{eq.MLMCN0} and \eqref{eq.MLMCNl} such that $\| Z - \mathbb{E} [ f(x(1)) ] \| \leq \varepsilon$ holds.

It remains to analyze the computational cost $C_{\mbox{\tiny MLMC}}$ of the MLMC estimator \eqref{eq.MLMCFEM}. Actually, the fact that the fast Euler method \eqref{eq.FEM} sampling a path needs a computational cost of $\mathcal{O}\big( |\ln h_l|^2 h_l^{-1} \big)$ (see \cite[Theorem 4.4]{FangLi2020}) indicates
\begin{align*}
C_{\mbox{\tiny MLMC}} &\leq C N_0 + \sum_{l=1}^{L} N_l \mathcal{O}\big( |\ln h_l|^2 h_l^{-1} \big) \\
&\leq C \Big( 3 C_2 \varepsilon^{-2} + 1 \Big) + \sum_{l=1}^{L} \Big( 3 C_3 \varepsilon^{-2} h_{l}^{\frac{5-4H}{2}} |\ln h_{l}|^2 \sum_{l'=1}^{L} h_{l'}^{\frac{3-4H}{2}} + 1\Big) C |\ln h_l|^2 h_l^{-1} \\
&\leq C \varepsilon^{-2} + C \varepsilon^{-2} \sum_{l=1}^{L} |\ln h_{l}|^4 h_{l}^{\frac{3-4H}{2}} \sum_{l'=1}^{L} h_{l'}^{\frac{3-4H}{2}} \\
&\leq
\begin{cases}
C \varepsilon^{-2}, & \mbox{if } H \in (\frac{1}{2}, \frac{3}{4}), \\
C \varepsilon^{-2} |\ln \varepsilon|^6, & \mbox{if } H = \frac{3}{4}, \\
C \varepsilon^{-(2+\frac{4H-3}{2-2H-\rho})} |\ln \varepsilon|^4, & \mbox{if } H \in (\frac{3}{4}, 1),
\end{cases}
\end{align*}
where the last step used the subsequent Lemma \ref{lm.MLMCforhL} and the following relations
\begin{align*}
& L = \log_M \varepsilon^{-\frac{1}{2-2H-\rho}} + C \leq C | \ln \varepsilon |, \\
& h_L = M^{-L} = C \varepsilon^{\frac{1}{2-2H-\rho}} \quad \mbox{and} \quad |\ln h_L| \leq C | \ln \varepsilon |.
\end{align*}
The proof is completed.
\end{proof}

\begin{lemma} \label{lm.MLMCforhL}
Let $M \geq 2$, $h_l = M^{-l} \ (l =1,2,\cdots,L)$, $\beta \geq 0$ and $\gamma \in \mathbb{R}$. Then, there exists some positive constant $C$ $($only depends on $M, \beta, \gamma$, but not on $L)$ such that
\begin{align*}
\sum_{l=1}^{L} |\ln h_l|^{\beta} h_l^{\gamma} \leq
\begin{cases}
C, & \quad\mbox{if } \gamma > 0, \\
|\ln h_L|^{\beta} L, & \quad\mbox{if } \gamma = 0, \\
C |\ln h_L|^{\beta} h_L^{\gamma}, & \quad\mbox{if } \gamma < 0.
\end{cases}
\end{align*}
\end{lemma}

\begin{proof}
When $\gamma > 0$, it follows from $M \geq 2 > 1$ that
\begin{align*}
\sum_{l=1}^{L} |\ln h_l|^{\beta} h_l^{\gamma} &= \sum_{l=1}^{L} \big( l \ln M \big)^{\beta} M^{-\gamma l} \leq C \sum_{l=1}^{+\infty} l^{\beta} M^{-\gamma l} \leq C.
\end{align*}
Note that the result is trivial for the case $\gamma = 0$. While $\gamma < 0$, the fact that $|\ln h_l|$ is increasing with respect to $l$ shows
\begin{align*}
\sum_{l=1}^{L} |\ln h_l|^{\beta} h_l^{\gamma} \leq |\ln h_L|^{\beta} \sum_{l=1}^{L} M^{(L - l) \gamma} M^{-L\gamma} \leq |\ln h_L|^{\beta} h_L^{\gamma} \sum_{l=1}^{+\infty} M^{\gamma l} \leq C |\ln h_L|^{\beta} h_L^{\gamma},
\end{align*}
which completes the proof.
\end{proof}

\section{Preliminaries for the proof of Theorem \ref{thm.mainEM}} \label{sec.Preli}

To facilitate the proof of Theorem \ref{thm.mainEM}, this section is devoted to establishing the Malliavin differentiability of the exact solution and the sharp estimates of some singular integrals. In the following, the Beta function will be frequently used, which is defined by
\begin{align} \label{eq.Betafun}
B(a,b) := \int_0^1 u^{a-1} (1-u)^{b-1} \mathrm d u, \qquad \mbox{for } a,b > 0.
\end{align}

\subsection{Malliavin calculus with respect to the fBm} \label{sec.MalliavinfBm}

Let us start with some basic definitions and Malliavin calculus with respect to the fBm; see \cite{Nualart2006,HongHuang2020} for more details. Throughout this paper, we always assume $H \in (1/2,1)$, and in this case, the covariance of the fBm $\{W_H(t)\}_{t\in[0,T]}$ possesses the following form
\begin{align*}
{\rm Cov}(t,s) = H(2H-1) \int_{0}^{t}\int_{0}^{s} |u-v|^{2H-2} \mathrm{d}u \mathrm{d}v,\quad \forall\, s,t \in [0,T].
\end{align*}
Denote by $\mathcal{E}$ the set of real-valued step functions on $[0,T]$ and let $\mathcal H$ be the Hilbert space defined as the closure of $\mathcal E$ with respect to the inner product $\langle \mathbf{1}_{[0,t]} (\cdot), \mathbf{1}_{[0,s]} (\cdot) \rangle_{\mathcal{H}} := {\rm Cov}(t,s)$. The map $\mathbf 1_{[0,t]} (\cdot) \mapsto W_H(t)$ can be extended to an isometry between $\mathcal{H}$ and a closed subspace of $L^2(\Omega,\mathcal F, \mathbb P)$. More precisely, denote this isometry by $\varphi \mapsto W_H(\varphi)$, and then for any $\varphi,\phi \in \mathcal{H}$, 
\begin{align*}
\langle \varphi, \phi \rangle_{\mathcal{H}} = H(2H-1) \int_{0}^{T}\int_{0}^{T} \varphi(u) \phi(v) |u-v|^{2H-2} \mathrm{d}u \mathrm{d}v = \mathbb{E}\Big[ W_H(\varphi) W_H(\phi) \Big].
\end{align*}

Denote by $\mathcal{S}$ the class of smooth real-valued random variables such that $F \in \mathcal{S}$ has the form
\begin{align*}
F = f(W_H(\varphi_1),\ldots,W_H(\varphi_n)),
\end{align*}
where $f \in C_p^\infty(\mathbb{R}^n;\mathbb{R})$, $\varphi_i \in \mathcal{H}, \, i=1,\ldots,n,\, n \in \mathbb{N_+}$. Here, $C_p^\infty(\mathbb{R}^n;\mathbb{R})$ is the space of all real-valued smooth functions on $\mathbb{R}^n$ with polynomial growth. The Malliavin derivative of $F \in \mathcal{S}$ is an $\mathcal{H}$-valued random variable defined by
\begin{align*}
D F = \sum_{i=1}^n \frac{\partial}{\partial x_i} f(W_H(\varphi_1),\ldots,W_H(\varphi_n)) \varphi_i,
\end{align*}
which is also a stochastic process $DF = \{ D_r F \}_{r \in [0,T]}$ with 
$$D_r F = \sum_{i=1}^n \frac{\partial}{\partial x_i} f(W_H(\varphi_1),\ldots,W_H(\varphi_n)) \varphi_i(r).$$ 

For any $q \geq1$, we denote the domain of $D$ in $L^q(\Omega;\mathbb{R})$ by $\mathbb{D}^{1,q}$, meaning that $\mathbb{D}^{1,q}$ is the closure of $\mathcal{S}$ with respect to the norm
\begin{align*}
\|F\|_{\mathbb{D}^{1,q}} = \Big( \mathbb{E} \big[ |F|^q \big] + \mathbb{E} \big[ \|DF\|_{\mathcal{H}}^q \big] \Big)^{\frac{1}{q}}.
\end{align*}
In a similar manner, for $F \in \mathcal{S}$, the iterated derivative $D^k F$ ($k \in \mathbb{N}_+$) is defined as a random variable with values in $\mathcal{H}^{\otimes k}$. For every $q \geq 1$ and $k \in \mathbb{N}_+$, denote by $\mathbb{D}^{k,q}$ the completion of $\mathcal{S}$ with respect to the norm
\begin{align*}
\|F\|_{\mathbb{D}^{k,q}} = \left( \mathbb{E}\Big[ |F|^q + \sum_{j=1}^{k} \|D^jF\|_{\mathcal{H}^{\otimes j}}^q \Big] \right)^{\frac{1}{q}}.
\end{align*}
We also denote $\mathbb{D}^{k,\infty} := \bigcap_{q \in [1,\infty)} \mathbb{D}^{k,q}$ for simplicity. For $F \in \mathbb{D}^{2,\infty}$ and $\phi \in C_{p}^{2}$, it follows from \cite[Remark 3.4]{Sanz-Sole2005} that $\phi(F) \in \mathbb{D}^{2,\infty}$, $D(\phi(F)) = \phi'(F) D F$ and $D^2(\phi(F)) = \phi''(F) D F \otimes D F + \phi'(F) D^2 F$.

Next, we introduce the adjoint operator $\delta$ of the derivative operator $D$, which is also known as the Skorohod integral. If an $\mathcal{H}$-valued random variable $\varphi\in L^2(\Omega;\mathcal{H})$ satisfies
\begin{align*}
\big| \mathbb{E}[\langle \varphi, DF \rangle_{\mathcal{H}}] \big| \leq C(\varphi) \| F \|, \quad \forall\, F\in \mathbb{D}^{1,2},
\end{align*}
then $\varphi\in {\rm Dom}(\delta)$ and $\delta(\varphi)\in L^2(\Omega;\mathbb{R})$ is characterized by the dual formula
\begin{align*}
\mathbb{E}[ \langle \varphi, DF \rangle_{\mathcal{H}}] = \mathbb{E}[F\delta(\varphi)], \quad \forall\, F\in \mathbb{D}^{1,2}.
\end{align*}
In particular, when $\varphi \in \mathcal{H}$ is deterministic, the Skorohod integral $\delta(\varphi)$ coincides with the Riemann--Stieltjes integral $\int_0^T \varphi(u) \mathrm{d} W_H(u)$.

\subsection{Malliavin differentiability of the exact solution}

The following extension of Gr\"{o}nwall's lemma is established in \cite[Lemma 15]{Dalang1999}.

\begin{lemma} \label{lem.Gronwall}
Let $\{ f_n \}_{ n\in\mathbb N}$ be a sequence of nonnegative real-valued functions on $[0,T]$ and $k_1$, $k_2$ be nonnegative real numbers such that for all $n \geq 1$, $t \in [0,T]$ and some $g \in L^1 ([0,T]; \mathbb{R}_+)$,
\begin{align*}
f_n(t) \leq k_1+\int_0^t \big( k_2 + f_{n-1}(s) \big) g(t-s) \mathrm{d} s.
\end{align*}
Then, there is a sequence $\{ a_n \}_{ n\in\mathbb N_+}$ of nonnegative real numbers satisfying $\sum_{n=1}^{\infty} a_n < \infty$ with the property:\
If $\sup_{t \in [0,T]} f_0(t) = M$, then for all $n \geq 1$, $t \in [0,T]$,
\begin{equation*}
f_n(t)\leq k_1 + (k_1+k_2) \sum_{i=1}^{n-1} a_i + (k_2 + M) a_n,
\end{equation*}
which implies $sup_{n\geq 0}\sup_{t \in [0,T]} f_n(t) < \infty$. Moreover, if further $k_1=k_2=0$, then $\sum_{n\geq 0}f_n(t)$ converges uniformly on $[0,T]$.
\end{lemma}

Now, we consider the first and second order Malliavin derivatives of the exact solution in Theorem \ref{thm.Malliavin}, which turns out to be bounded by some quantities related to the singular kernel in the model \eqref{eq.GLE}, rather than by a constant, compared to the case of SDEs with fBm (see e.g., \cite{KloedenNeuenkirch2011}). 

\begin{theorem} \label{thm.Malliavin}
Let $H \in (1/2, 1)$ and $\alpha \in (1-H,1)$. If $b \in C_{b,p}^{1,2}$, then $x(t) \in \mathbb{D}^{2,\infty}$ for any $t \in [0,T]$, and there exists some constant $C = C(\alpha,H,\sigma,T)$ such that for $r\in[0,T]$,
\begin{align}\label{Drxt}
| D_r x(t) | \leq C (t - r)^{\alpha-1} \mathbf{1}_{[0,t)}(r),\quad \mbox{a.s.}
\end{align}
Moreover, for any $q\geq1$, there exists some constant $C = C(\alpha,H,\sigma,T,q)$ such that for $r_1, r_2 \in [0,T]$, 
\begin{align}\label{Drxt-2}
\| D_{r_2} D_{r_1} x(t) \|_q \leq C \int_{r_1 \vee r_2}^t (t-s)^{\alpha-1} (s-r_1)^{\alpha-1} (s-r_2)^{\alpha-1} \mathrm{d} s \mathbf{1}_{[0,t)}(r_1 \vee r_2).
\end{align}
\end{theorem}

\begin{proof}
Consider the following standard Picard iteration sequence
\begin{align}\label{eq.Picard}
x^{(n+1)}(t) = x_0 + \frac{1}{\Gamma(\alpha)} \int_0^t (t-s)^{\alpha-1} b \big( x^{(n)}(s) \big) \mathrm{d} s + G(t), \quad n \geq 0
\end{align}
with $x^{(0)}(t) = x_0$ for all $t \in [0,T]$. To prove the statement $x(t) \in \mathbb{D}^{2,\infty}$ for all $t \in [0,T]$, it suffices to show that $\{ x^{(n)}(t) \}_{n=0}^{\infty}$ converges to $x(t)$ in $L^q(\Omega;\mathbb{R})$ and $\sup_{n \geq 0} \| x^{(n)}(t) \|_{\mathbb{D}^{2,q}} < \infty$ for any $q \geq 1$, in view of \cite[Lemma 1.5.3]{Nualart2006}.

\textit{Claim 1:\ $\{ x^{(n)}(t) \}_{n=0}^{\infty}$ converges to $x(t)$ in $L^q(\Omega;\mathbb{R})$ for any $q \geq 1$.}

Let $q\ge1$. Since the law of $G(t)$ is Gaussian, it follows from \cite[Lemma 2]{LiLiu2017} that for any $t \in [0,T]$, 
\begin{align} \label{eq.t^alpha-1}
\|G(t)\|_q \leq C(q)\|G(t)\| \leq C \|(t-\cdot)^{\alpha-1}\mathbf{1}_{[0,t)}(\cdot)\|_{\mathcal H} \leq C
\end{align}
with $C = C(\alpha,H,T,\sigma,q)$. Then, by \eqref{eq.Picard} and \eqref{eq.t^alpha-1},
\begin{align*}
\|x^{(1)}(t) - x^{(0)}(t) \|_q = \|b(x_0)\frac{1}{\Gamma(\alpha+1)}t^\alpha+G(t)\|_q \leq C(\alpha,H, T,\sigma,q,x_0).
\end{align*}
For each $n \geq 1$, the assumption $b \in C_b^1$ reads
\begin{align*}
\|x^{(n+1)}(t) - x^{(n)}(t) \|_q \leq C \int_0^t(t-s)^{\alpha-1}\|x^{(n)}(s) - x^{(n-1)}(s) \|_q \mathrm{d} s.
\end{align*}
Applying Lemma \ref{lem.Gronwall} to the sequence $\{f_n\}_{n \geq 0}$ with $f_n(t) := \|x^{(n+1)}(t) - x^{(n)}(t) \|_q$ shows
\begin{align*}
\sum_{n=0}^{\infty} \sup_{t\in[0,T]}\|x^{(n+1)}(t) - x^{(n)}(t) \|_q < \infty,
\end{align*}
which implies that $x^{(n)}(t) = x^{(0)}(t) + \sum_{k=0}^{n-1}\big( x^{(k+1)}(t) - x^{(k)}(t) \big)$ converges uniformly in $L^q(\Omega;\mathbb{R})$ with respect to $t \in [0,T]$, and
\begin{align} \label{eq.XnLq}
\sup_{t\in[0,T]} \|x^{(n)}(t)\|_q < \infty.
\end{align}
Putting $n\rightarrow \infty$ on both sides of \eqref{eq.Picard}, the limit of $\{x^{(n)}(t)\}_{n \geq 0}$ satisfies Eq.\! \eqref{eq.GLE}, and
\begin{align} \label{eq.XtLq}
\sup_{t\in[0,T]} \|x(t)\|_q < \infty.
\end{align}

\textit{Claim 2:\ $\sup_{n \geq 0} \| x^{(n)}(t) \|_{\mathbb{D}^{2,q}} < \infty$ for any $q \geq 1$.}

Obviously, $x^{(0)}(t) \in \mathbb{D}^{2,\infty}$. Assume by induction that $x^{(n)}(t) \in \mathbb D^{2,\infty}$, it follows from $b \in C_{b,p}^{1,2}$ that $x^{(n+1)}(t) \in \mathbb D^{2,\infty}$. Thus, we have $x^{(n)}(t) \in \mathbb D^{2,\infty}$, for any integer $n \geq 1$. Taking the Malliavin derivative on both sides of \eqref{eq.Picard}, the chain rule of the Malliavin derivative gives that for $r<t$,
\begin{align}\label{eq.DrXn}
D_r x^{(n+1)}(t) &= \frac{1}{\Gamma(\alpha)} \int_0^t (t-s)^{\alpha-1} b^{\prime} (x^{(n)}(s)) D_rx^{(n)}(s) \mathbf{1}_{[0,s)}(r) \mathrm{d} s + \frac{\sigma}{\Gamma(\alpha)}(t-r)^{\alpha-1}, 
\end{align}
and for $r\in[t,T]$, $D_rx^{(n+1)}(t)=0$. Then, using \eqref{eq.t^alpha-1}, \eqref{eq.DrXn} and the assumption $b \in C_b^1$, we have
\begin{align*}
\| Dx^{(n+1)}(t) \|_{L^q(\Omega;\mathcal H)} &\leq C\int_0^t (t-s)^{\alpha-1} \|Dx^{(n)}(s) \|_{L^q(\Omega;\mathcal H)} \mathrm{d} s+C\|(t-\cdot)^{\alpha-1}\mathbf{1}_{[0,t)}(\cdot)\|_{\mathcal H} \\
&\leq C\int_0^t (t-s)^{\alpha-1} \| Dx^{(n)}(s) \|_{L^q(\Omega;\mathcal H)} \mathrm{d} s + C(\alpha,H,T,\sigma).
\end{align*}
Applying Lemma \ref{lem.Gronwall} to the sequence $\{f_n\}_{n \geq 0}$ with $f_n(t) := \| Dx^{(n)}(t) \|_{L^q(\Omega;\mathcal H)}$ obtains
\begin{align}\label{eq.DrXnLqBound}
\sup_{n\ge0}\sup_{t\in[0,T]}\| Dx^{(n)}(t) \|_{L^q(\Omega;\mathcal H)} \leq C(\alpha,H,T,\sigma), \quad \forall \, q \geq 1.
\end{align}

Taking the Malliavin derivative on both sides of \eqref{eq.DrXn}, we also have for $r_1, \, r_2 \in [0,T]$,
\begin{align*}
D_{r_2} D_{r_1} x^{(n+1)}(t)
&= \frac{1}{\Gamma(\alpha) }\int_0^t (t-s)^{\alpha-1} b^{\prime\prime}(x^{(n)}(s)) D_{r_2} x^{(n)}(s) D_{r_1} x^{(n)}(s) \mathbf{1}_{[0,s)}(r_1 \vee r_2) \mathrm{d} s\\
&\quad + \frac{1}{\Gamma(\alpha)} \int_0^t (t-s)^{\alpha-1} b^{\prime}(x^{(n)}(s)) D_{r_2} D_{r_1} x^{(n)}(s) \mathbf{1}_{[0,s)}(r_1 \vee r_2) \mathrm{d} s, 
\end{align*}
Then, using the assumption $b \in C_{b,p}^{1,2}$, H\"older's inequality as well as \eqref{eq.XnLq} and \eqref{eq.DrXnLqBound}, one can derive
\begin{align*}
\|D^2 x^{(n+1)}(t) \|_{L^q(\Omega;\mathcal H^{\otimes 2})} 
&\leq C \int_0^t (t-s)^{\alpha-1} \| b^{\prime\prime}(x^{(n)}(s)) \|_{2q} \| D x^{(n)}(s) \|_{L^{4q}(\Omega;\mathcal H)}^2 \mathrm{d} s \\
&\quad + C \int_0^t (t-s)^{\alpha-1} \| D^2 x^{(n)}(s) \|_{L^{q}(\Omega;\mathcal H^{\otimes 2})}\mathrm{d} s\\
&\leq C + C \int_0^t (t-s)^{\alpha-1 }\| D^2 x^{(n)}(s) \|_{L^{q}(\Omega;\mathcal H^{\otimes 2})} \mathrm{d} s.
\end{align*}
By means of Lemma \ref{lem.Gronwall}, a similar argument as in the proof of \eqref{eq.DrXnLqBound} will yield
\begin{align*} 
\sup_{n\ge0} \sup_{t\in[0,T]} \|D^2 x^{(n)}(t)\|_{L^q(\Omega;\mathcal H^{\otimes 2})} \leq C(\alpha,H,T,\sigma),
\end{align*}
which together with \eqref{eq.DrXnLqBound} proves \textit{Claim 2}. Therefore, $x(t) \in \mathbb{D}^{2,\infty}$ for all $t \in [0,T]$.

It remains to prove the estimates \eqref{Drxt} and \eqref{Drxt-2}. By the chain rule of the Malliavin derivative, we have that for $r \in [0,T]$,
\begin{align}\label{eq.DrXt}
D_r x(t) = \Big( \frac{1}{\Gamma(\alpha)} \int_r^t (t-s)^{\alpha-1 }b^{\prime}(x(s)) D_r x(s) \mathrm{d} s + \frac{\sigma}{\Gamma(\alpha)} (t-r)^{\alpha-1} \Big) \mathbf{1}_{[0,t)}(r).
\end{align}
Therefore,
\begin{align*}
|D_r x(t)| \leq C \Big( \int_r^t (t-s)^{\alpha-1} |D_r x(s)| \mathrm{d} s + (t-r)^{\alpha-1} \Big) \mathbf{1}_{[0,t)}(r),
\end{align*}
from which one sees that \eqref{Drxt} is a direct consequence of the singular Gr\"{o}nwall inequality (see e.g., \cite[Lemma 7.4]{ZhengWang2020}). Now we turn to proving \eqref{Drxt-2}. Taking the Malliavin derivative on both sides of \eqref{eq.DrXt}, and taking the estimates \eqref{Drxt} and \eqref{eq.XtLq} into account, one obtains that for $r_1, \, r_2 \in [0,T]$, $q\geq1$,
\begin{align*}
\| D_{r_2} D_{r_1} x(t) \|_q
&\leq C \bigg( \int_{r_1 \vee r_2}^t (t-s)^{\alpha-1} (s-r_1)^{\alpha-1} (s-r_2)^{\alpha-1} \mathrm{d} s \\
&\quad\ + \int_{r_1 \vee r_2}^t (t-s)^{\alpha-1} \| D_{r_2} D_{r_1} x(s) \|_q \mathrm{d} s \bigg)\mathbf{1}_{[0,t)}(r_1 \vee r_2).
\end{align*}
Then by the singular Gr\"{o}nwall inequality (see e.g., \cite[Lemma 7.4]{ZhengWang2020}), one gets
\begin{align*}
&~\| D_{r_2} D_{r_1} x(t) \|_q \\
\leq &~C \mathbf{1}_{[0,t)}(r_1 \vee r_2) \Big[\int_{r_1 \vee r_2}^t (t-s)^{\alpha-1} (s-r_1)^{\alpha-1} (s-r_2)^{\alpha-1} \mathrm{d} s + \\
&\int_{r_1 \vee r_2}^t \sum_{n=1}^{\infty}\frac{(C \Gamma(\alpha))^n}{\Gamma(n\alpha)}(t-s)^{n\alpha-1}
\int_{r_1 \vee r_2}^s(s-u)^{\alpha-1}(u-r_1)^{\alpha-1}(u-r_2)^{\alpha-1}\mathrm{d} u \mathrm{d} s \Big].
\end{align*}
Notice that by the Fubini theorem and the relation $B(n\alpha,\alpha)=\frac{\Gamma(n\alpha)\Gamma(\alpha)}{\Gamma(n\alpha+\alpha)}$,
\begin{align*}
&\quad\ \int_{r_1 \vee r_2}^t \sum_{n=1}^{\infty} \frac{(C \Gamma(\alpha))^n}{\Gamma(n\alpha)} (t-s)^{n\alpha-1}
\int_{r_1 \vee r_2}^s(s-u)^{\alpha-1}(u-r_1)^{\alpha-1}(u-r_2)^{\alpha-1}\mathrm{d} u \mathrm{d} s \\
&= \int_{r_1 \vee r_2}^t \sum_{n=1}^{\infty} \frac{(C \Gamma(\alpha))^n}{\Gamma(n\alpha)} \left( \int_u^t (t-s)^{n\alpha-1} (s-u)^{\alpha-1} \mathrm{d} s \right) (u-r_1)^{\alpha-1} (u-r_2)^{\alpha-1} \mathrm{d} u \\
&= \Gamma(\alpha) \int_{r_1 \vee r_2}^t \sum_{n=1}^{\infty} \frac{\big( C \Gamma(\alpha) (t-u)^\alpha \big)^n}{\Gamma(n\alpha+\alpha)} (t-u)^{\alpha-1} (u-r_1)^{\alpha-1} (u-r_2)^{\alpha-1} \mathrm{d} u \\
&\leq C E_{\alpha,\alpha} (C \Gamma(\alpha) T^\alpha) \int_{r_1 \vee r_2}^t (t-u)^{\alpha-1} (u-r_1)^{\alpha-1} (u-r_2)^{\alpha-1} \mathrm{d} u,
\end{align*}
where $E_{\alpha,\alpha}(z) = \sum_{n=0}^\infty \frac{z^n}{\Gamma(\alpha n+\alpha)}$, $z\in\mathbb{R}$ is the Mittag--Leffler function, and the exchange of the integrals is justified since the integrand is nonnegative. Gathering the above estimates together proves \eqref{Drxt-2}. The proof is completed.
\end{proof}

\subsection{Estimates of singular integrals}

In this part, we give some estimates of singular integrals, which are vital to the proof of Theorem \ref{thm.mainEM}.

\begin{lemma} 
Let $H \in (1/2, 1)$ and $\alpha \in (1-H,1)$. Then there exists some constant $C$ such that for any $1 \leq i < j \leq N$ and $s \in (t_{i-1}, t_i]$,
\begin{gather}
\int_0^{t_{i-1}} \int_{t_{i-1}}^s (s-u)^{\alpha-1} (u-v)^{2H-2} (t_{j-1}-v)^{\alpha-2H} \mathrm{d}u \mathrm{d} v \leq Ch^{2\alpha-1}, \label{eq.lem5.1-2} \\
\int_0^{t_{i-1}} \int_{v}^{t_{i-1}} \big( (t_{i-1}-u)^{\alpha - 1} - (s-u)^{\alpha - 1} \big) (t_{j-1}-v)^{\alpha - 2H} (u-v)^{2H-2} \mathrm{d}u\mathrm{d}v\leq Ch^{2\alpha-1}. \label{eq.lem5.1-1}
\end{gather}
\end{lemma}

\begin{proof}
For $\beta \in (-2H,1-2H)$, using \eqref{eq.Betafun} yields
\begin{align}\label{eq.beta}\notag
&\quad\ \int_0^{t_{i-1}} \int_{t_{i-1}}^s (s-u)^{\alpha-1} (u-v)^{2H-2} (t_{j-1}-v)^{\beta} \mathrm{d}u \mathrm{d} v\\
& \leq \int_{t_{i-1}}^s \int_0^{t_{i-1}} (s-u)^{\alpha - 1} (u-v)^{2H-2+\beta} \mathrm{d}v \mathrm{d}u
\leq C (s - t_{i-1})^{\alpha+2H-1+\beta}.
\end{align}
We proceed to prove \eqref{eq.lem5.1-2} by separating the cases $2\alpha\geq1$ and $2\alpha<1$. When $2\alpha\geq1$, applying $s - t_{i-1} \leq h$ and \eqref{eq.beta} with $\beta = \alpha-2H$ yields \eqref{eq.lem5.1-2}. While $2\alpha<1$, by $t_{j-1}-v \geq h$, for $v \in [0,t_{i-1}]$, we have
\begin{align*}
&\quad\ \int_0^{t_{i-1}} \int_{t_{i-1}}^s (s-u)^{\alpha-1} (u-v)^{2H-2} (t_{j-1}-v)^{\alpha-2H} \mathrm{d}u \mathrm{d} v \\
&= \int_{t_{i-1}}^s \int_0^{t_{i-1}} (s-u)^{\alpha - 1} (u-v)^{2H-2} (t_{j-1}-v)^{2\alpha - 1} (t_{j-1}-v)^{1-\alpha - 2H} \mathrm{d}v \mathrm{d}u \\
&\leq h^{2\alpha-1} \int_{t_{i-1}}^s \int_0^{t_{i-1}} (s-u)^{\alpha - 1} (u-v)^{2H-2} (t_{j-1}-v)^{1-\alpha - 2H} \mathrm{d}v \mathrm{d}u \leq C h^{2\alpha-1},
\end{align*}
thanks to \eqref{eq.beta} with $\beta=1-\alpha - 2H \in (-2H,1-2H)$.

Now we turn to the proof of \eqref{eq.lem5.1-1}. Note that for any $\theta\in(-1,0)$, there exists $C=C(\theta)$ such that
\begin{align}\label{chi}
\int_0^{u} (u-r)^{\theta} - (t-r)^{\theta}\mathrm{d} r \leq C(t-u)^{1+\theta},\quad \forall~0\leq u<t\leq T.
\end{align}
We also claim
\begin{align}\label{eq.lem5.1-3}
\int_0^{t_{i-1}} \big( (t_{i-1}-v)^{\alpha+2H-2} - (s-v)^{\alpha+2H-2} \big) (t_{j-1}-v)^{\alpha - 2H} \mathrm{d}v \leq C h^{2\alpha - 1}.
\end{align}
Indeed, for $\alpha \in [2-2H,1)$, \eqref{eq.lem5.1-3} holds trivially since $t_{i-1} < s$ and $\alpha+2H-2 \geq 0$. For $\alpha \in (1-H,2-2H)$, using $t_{j-1}-v \geq h$, for $v \in [0,t_{i-1}]$, yields
\begin{align*}
&\quad\ \int_0^{t_{i-1}} \big( (t_{i-1}-v)^{\alpha+2H-2} - (s-v)^{\alpha+2H-2} \big) (t_{j-1}-v)^{\alpha - 2H} \mathrm{d}v \\
&\leq h^{\alpha - 2H} \int_0^{t_{i-1}} (t_{i-1}-v)^{\alpha+2H-2} - (s-v)^{\alpha+2H-2} \mathrm{d}v \leq C h^{2\alpha - 1},
\end{align*}
thanks to \eqref{chi} with $\theta = \alpha+2H-2 \in (-1,0)$, which proves \eqref{eq.lem5.1-3}. Using \eqref{eq.Betafun} again,
\begin{align}\label{eq.Beta}
&\quad\ \int_{v}^{t_{i-1}} \big( (t_{i-1}-u)^{\alpha - 1} - (s-u)^{\alpha - 1} \big) (u-v)^{2H-2} \mathrm{d}u \\
&= B(\alpha,2H-1) \Big( (t_{i-1}-v)^{\alpha+2H-2} - (s-v)^{\alpha+2H-2} \Big) + \int_{t_{i-1}}^s (s-u)^{\alpha-1} (u-v)^{2H-2} \mathrm{d}u.\notag
\end{align}
Finally, \eqref{eq.lem5.1-1} follows from \eqref{eq.lem5.1-2}, \eqref{eq.lem5.1-3} and \eqref{eq.Beta}. The proof is completed.
\end{proof}

\begin{lemma}\label{eq.ijts}
Let $H \in (1/2, 1)$ and $\alpha \in (1-H,1)$. Then there exists some constant $C$ such that for any $0 < u \leq t_{k-1} < s \leq t_k$ and $ \tau \in(t_{l-1},t_l]$ with $1 \leq k < l \leq N$, 
\begin{align*}
\int_{t_{k-1}}^{t_{l-1}}
\big( (t_{l-1}-v)^{\alpha - 1} - (\tau-v)^{\alpha - 1} \big) |v-u|^{2H-2} \mathrm{d}v\leq C h^{\alpha} (\tau -s)^{2H-2}.
\end{align*}
\end{lemma}

\begin{proof}
\emph{Case 1: $k < l \leq k+2$.} In this case, $\tau - s \leq 3h$. Then, by \eqref{eq.Betafun},
\begin{align*}
&\quad\ \int_{t_{k-1}}^{t_{l-1}}
\big( (t_{l-1}-v)^{\alpha - 1} - (\tau-v)^{\alpha - 1} \big) |v-u|^{2H-2} \mathrm{d}v\\
&\leq\int_{t_{k-1}}^{t_{l-1}}(t_{l-1}-v)^{\alpha - 1} (v-t_{k-1})^{2H-2} \mathrm{d}v
\leq C h^{\alpha + 2H -2 }
\leq C h^{\alpha} (\tau-s)^{2H-2}.
\end{align*}

\emph{Case 2: $k+3 \leq l \leq N.$} In this case, $t_{k-1} \leq \frac{\tau+s}{2} \leq t_{l-1} - \frac{1}{2}h$. For $v\in(t_{k-1},\frac{\tau+s}{2})$, we obtain
\begin{align*}
\tau -v \leq (l-k+1)h \leq 4(l-k-2)h = 8(t_{l-1} - \frac{t_l + t_k}{2}) 
\leq 8(t_{l-1} -v),
\end{align*}
which implies $\tau-s \leq 2(\tau -v) \leq 16(t_{l-1}-v)$ and
\begin{align*}
(t_{l-1}-v)^{\alpha - 1} - (\tau-v)^{\alpha - 1}&\leq (1-\alpha)(t_{l-1}-v)^{\alpha - 2}h \leq Ch(t_{l-1}-v)^{\alpha - 2H}(\tau-s)^{2H-2}.
\end{align*}
Hence, it follows from $u\leq t_{k-1}$ and $t_{l-1} - v \geq \frac{1}{2} h$ for $v \in (t_{k-1}, \frac{\tau+s}{2})$ that
\begin{align*}
&\quad\ \int_{t_{k-1}}^{\frac{\tau+s}{2}} \big( (t_{l-1}-v)^{\alpha - 1} - (\tau-v)^{\alpha - 1} \big) |v-u|^{2H-2} \mathrm{d}v \\
& \leq Ch\int_{t_{k-1}}^{\frac{\tau+s}{2}} (t_{l-1}-v)^{\alpha - 2H} (v-u)^{2H-2} \mathrm{d}v (\tau -s)^{2H-2} \\
&\leq C h^{\alpha} \int_{t_{k-1}}^{t_{l-1}} (t_{l-1}-v)^{1 - 2H} (v-t_{k-1})^{2H-2} \mathrm{d}v (\tau -s)^{2H-2} \leq C h^{\alpha} (\tau -s)^{2H-2}.
\end{align*}
For $v \in (\frac{\tau+s}{2}, t_{l-1})$, we have $\tau-s \leq 2(v-s)\leq 2(v-u)$, for $u\leq t_{k-1}<s$, and then
\begin{align*}
&\quad\ \int_{\frac{\tau+s}{2}}^{t_{l-1}} \big( (t_{l-1}-v)^{\alpha - 1} - (\tau-v)^{\alpha - 1} \big) |v-u|^{2H-2} \mathrm{d}v \\
&\leq C \int_{\frac{\tau+s}{2}}^{t_{l-1}} \big( (t_{l-1}-v)^{\alpha - 1} - (\tau-v)^{\alpha - 1} \big) \mathrm{d}v (\tau -s)^{2H-2} \leq C h^{\alpha} (\tau -s)^{2H-2},
\end{align*}
which completes the proof.
\end{proof}

\begin{lemma} \label{lem.midNext}
Assume either
\begin{align*}
\mbox{\emph{(C1)}}~\beta<0 \mbox{ and }\gamma<0 \mbox{ with } -1<\beta+\gamma<0 \quad \mbox{or }\quad \mbox{\emph{(C2)}}~ \beta\geq 0,~\gamma\in(-1,0).
\end{align*}
Then there exists some constant $C$ such that for any $\tau \in (t_{j-1},t_j]$ and $\eta \in (0,t_{j-1}]$ with $1< j \leq N$, 
\begin{align}\label{eq.gammabeta}
& \int_0^{\eta} (\eta-r)^{\beta} \big( (t_{j-1}-r)^{\gamma} - (\tau-r)^{\gamma} \big) \mathrm{d} r \leq C h^{(\beta+\gamma+1)\wedge(\gamma+1)}.
\end{align}
\end{lemma}

\begin{proof}
We first prove \eqref{eq.gammabeta} under condition (C1).
Notice that there exists a unique integer $k$ satisfying $t_{k-1} < \eta \leq t_k \leq t_{j-1}$.
Since $\gamma<0$, we have for all $r \in (0,t_{k-1}),$
\begin{align*}
&\quad\ (t_{j-1} - r)^{\gamma} - (\tau - r)^{\gamma} \leq (t_{j-1} - r)^{\gamma} - (t_j - r)^{\gamma}\le
(t_{k-1} - r)^{\gamma} - (t_{k} - r)^{\gamma}, 
\end{align*}
which together with $\beta<0$ and $\beta+\gamma\in(-1,0)$ indicates
\begin{align*}
&\quad\ \int_0^{t_{k-1}} (\eta -r)^{\beta} \big( (t_{j-1}-r)^{\gamma} - (\tau-r)^{\gamma} \big) \mathrm{d} r \\
&\leq \int_0^{t_{k-1}} (t_{k-1} -r)^{\beta} \big( (t_{k-1}-r)^{\gamma} - (t_k-r)^{\gamma} \big) \mathrm{d} r\\
&\leq \int_0^{t_{k-1}} (t_{k-1}-r)^{\beta+\gamma} - (t_k-r)^{\beta+\gamma} \mathrm{d} r\leq Ch^{\beta+\gamma+1},
\end{align*}
due to \eqref{chi} with $\theta=\beta+\gamma$.
In addition, since $t_{k-1}<\eta\leq t_{k}\leq t_{j-1}$ and $\gamma<0$,
\begin{align*}
&\quad\ \int_{t_{k-1}}^{\eta} (\eta -r)^{\beta} \big( (t_{j-1}-r)^{\gamma} - (\tau-r)^{\gamma} \big) \mathrm{d} r
\leq \int_{t_{k-1}}^{\eta} (\eta -r)^{\beta+\gamma} \mathrm{d} r \leq C h^{\beta+\gamma+1}.
\end{align*}
Thus, one can conclude that \eqref{eq.gammabeta} holds under condition (C1).

It remains to prove \eqref{eq.gammabeta} under condition (C2). Actually, observe that $\beta\ge0$ ensures $(\eta-r)^{\beta}\leq T^\beta$ for all $r\in[0,\eta]$. Hence, it follows from \eqref{chi} with $\theta=\gamma$ that
\begin{align*}
& \int_0^{\eta} (\eta-r)^{\beta} \big( (t_{j-1}-r)^{\gamma} - (\tau-r)^{\gamma} \big) \mathrm{d} r\leq C\int_0^{t_{j-1}}(t_{j-1}-r)^{\gamma} - (\tau-r)^{\gamma} \mathrm{d} r
 \leq C h^{\gamma+1}.
\end{align*}
The proof is complete.
\end{proof}

\begin{lemma} \label{lem.impLemUse}
Let $H \in (1/2,1)$ and $\alpha \in (1-H,1)$. Then there exists some constant $C$ such that for any $\eta \in [0,T]$ and $\tau \in (t_{j-1},t_j]$ with $1 \leq j \leq N$, 
\begin{gather}
\int_0^{\eta} \int_{t_{j-1}}^{\tau} (\eta-r)^{\alpha-1} (\tau-v)^{\alpha - 1} |v-r|^{2H-2} \mathrm{d} v \mathrm{d} r \leq C\mathcal{R}_{H, \alpha} (h), \label{eq.lem3.1} \\
\int_0^{\eta} \int_0^{t_{j-1}} (\eta-r)^{\alpha-1} \big( (t_{j-1}-v)^{\alpha - 1} - (\tau-v)^{\alpha - 1} \big) |v-r|^{2H-2} \mathrm{d} v \mathrm{d} r \leq C\mathcal{R}_{H, \alpha} (h), \label{eq.lem3.2}
\end{gather}
where $\mathcal{R}_{H, \alpha} (h)$ is defined in Theorem \ref{thm.mainEM}.
\end{lemma}

\begin{proof}
We will first prove \eqref{eq.lem3.1} with $1 < j \leq N$. Specifically, it is divided into the following four cases.

Case 1:\ $\eta \in [0, t_1]$. The estimate \eqref{eq.lem3.1} can be obtained by
\begin{align*}
&\quad\ \int_0^{\eta} \int_{t_{j-1}}^{\tau} (\eta-r)^{\alpha-1} (\tau-v)^{\alpha - 1} |v-r|^{2H-2} \mathrm{d} v \mathrm{d} r \\
&\leq \int_0^{\eta} \int_{r}^{\tau} (\eta-r)^{\alpha-1} (\tau-v)^{\alpha - 1} (v-r)^{2H-2} \mathrm{d} v \mathrm{d} r \\
&\leq C \int_0^{\eta} (\eta-r)^{\alpha-1} (\tau-r)^{\alpha+2H-2} \mathrm{d} r \leq C h^{\alpha \wedge 2(\alpha+H-1)}.
\end{align*}

Case 2:\ $\eta \in (t_1,t_{j-1}]$. In this case, we have
\begin{align*} 
&\quad\ \int_0^{\eta-h} \int_{t_{j-1}}^{\tau} (\eta-r)^{\alpha-1} (\tau-v)^{\alpha - 1} |v-r|^{2H-2} \mathrm{d} v \mathrm{d} r \\
&\leq \int_0^{\eta-h} \int_{t_{j-1}}^{\tau} (\eta-r)^{\alpha+2H-3} (\tau-v)^{\alpha - 1} \mathrm{d} v \mathrm{d} r \leq C\mathcal{R}_{H, \alpha} (h).
\end{align*}
Besides, in virtue of $\tau> t_{j-1}\ge\eta$,
\begin{align*}
&\quad\ \int_{\eta-h}^{\eta} \int_{t_{j-1}}^{\tau} (\eta-r)^{\alpha-1} (\tau-v)^{\alpha - 1} |v-r|^{2H-2} \mathrm{d} v \mathrm{d} r \\
&\leq C \int_{\eta-h}^{\eta} (\eta-r)^{\alpha-1} (\tau-r)^{\alpha+2H-2} \mathrm{d} r \leq C h^{\alpha \wedge 2(\alpha+H-1)}.
\end{align*}
Thus, the estimate \eqref{eq.lem3.1} holds for Case 2.

Case 3:\ $\eta \in (t_{j-1}, \tau]$.
In view of the result of Case 2, the left hand side of \eqref{eq.lem3.1} can be controlled by
\begin{align*}
&C\mathcal{R}_{H, \alpha} (h)
 +\int_{t_{j-1}}^{\eta} \int_{t_{j-1}}^{\tau} (\eta-r)^{\alpha-1} (\tau-v)^{\alpha - 1} |r-v|^{2H-2} \mathrm{d} v \mathrm{d} r.
\end{align*}
The second term can be bounded as
\begin{align} \label{eq.3intgrand}
& \int_{t_{j-1}}^{\eta} \int_{t_{j-1}}^{\tau} (\eta-r)^{\alpha-1} (\tau-v)^{\alpha - 1} |r-v|^{2H-2} \mathrm{d} v \mathrm{d} r \notag\\
=& \int_{t_{j-1}}^{\eta} \int_{t_{j-1}}^{r} (\eta-r)^{\alpha-1} (\tau-v)^{\alpha - 1} (r-v)^{2H-2} \mathrm{d} v \mathrm{d} r \notag\\
&+ \int_{t_{j-1}}^{\eta} \int_{r}^{\tau} (\eta-r)^{\alpha-1} (\tau-v)^{\alpha - 1} (v-r)^{2H-2} \mathrm{d} v \mathrm{d} r \nonumber\\
\leq& \int_{t_{j-1}}^{\eta} \int_{v}^{\eta} (\tau-v)^{\alpha - 1} (\eta-r)^{\alpha-1} (r-v)^{2H-2} \mathrm{d} r \mathrm{d} v \notag\\
&+ C \int_{t_{j-1}}^{\eta} (\eta-r)^{\alpha-1} (\tau-r)^{\alpha +2H-2} \mathrm{d} r \leq C h^{2(\alpha+H-1)}.
\end{align}
Hence, the estimate \eqref{eq.lem3.1} holds for Case 3.

Case 4:\ $\eta \in (\tau,T]$. Utilizing the result of Case 3, the estimate \eqref{eq.lem3.1} can be attained by
\begin{align*}
& \int_0^{\eta} \int_{t_{j-1}}^{\tau} (\eta-r)^{\alpha-1} (\tau-v)^{\alpha - 1} |v-r|^{2H-2} \mathrm{d} v \mathrm{d} r \\
=& \int_0^{\tau} \int_{t_{j-1}}^{\tau} (\eta-r)^{\alpha-1} (\tau-v)^{\alpha - 1} |v-r|^{2H-2} \mathrm{d} v \mathrm{d} r \\
&+ \int_{\tau}^{\eta} \int_{t_{j-1}}^{\tau} (\eta-r)^{\alpha-1} (\tau-v)^{\alpha - 1} (r-v)^{2H-2} \mathrm{d} v \mathrm{d} r \\
\leq& ~C\mathcal{R}_{H, \alpha} (h) + C \int_{t_{j-1}}^{\tau} (\eta-v)^{\alpha+2H-2} (\tau-v)^{\alpha - 1} \mathrm{d} v \\
\leq&~ C\mathcal{R}_{H, \alpha} (h) +C h^{\alpha \wedge 2(\alpha+H-1)} \leq C\mathcal{R}_{H, \alpha} (h).
\end{align*}

One can conclude that we have proven \eqref{eq.lem3.1} when $1 < j \leq N$. It remains to prove \eqref{eq.lem3.1} for the case $j = 1$. In this case, one can divide $\eta \in [0,T]$ into $\eta \in [0,\tau]$ and $\eta \in (\tau,T]$, which will fall into the above Cases 1 and 4, respectively. Thus, the proof of \eqref{eq.lem3.1} is completed.

Next, we turn to proving \eqref{eq.lem3.2} under the settings $\eta \in [0, t_{j-1}]$ and $\eta \in (t_{j-1},T]$, separately. When $\eta \in [0, t_{j-1}]$, it follows from Lemma \ref{lem.midNext} that
\begin{align} \label{eq.eta1}
 \int_0^{\eta} (\eta-r)^{\alpha+2H-2} \big( (t_{j-1}-r)^{\alpha - 1} - (\tau-r)^{\alpha - 1} \big) \mathrm{d} r
\leq C h^{\alpha \wedge 2(\alpha+H-1)}.
\end{align}
Similarly, for $\alpha < 2-2H$, Lemma \ref{lem.midNext} also implies that
\begin{align} \label{eq.eta2}
\int_0^{\eta} (\eta-r)^{\alpha-1} \big( (t_{j-1}-r)^{\alpha+2H-2} - (\tau-r)^{\alpha+2H-2} \big) \mathrm{d} r
\leq C h^{2(\alpha+H-1)},
\end{align}
and \eqref{eq.eta2} holds trivially for $\alpha \geq 2-2H$ since $t_{j-1} \leq \tau$. The left hand side of \eqref{eq.lem3.2} is equal to
\begin{align*}
\mathcal T_1 + \mathcal T_2 &:= \int_0^{\eta} \int_0^{r} (\eta-r)^{\alpha-1} \big( (t_{j-1}-v)^{\alpha - 1} - (\tau-v)^{\alpha - 1} \big) (r-v)^{2H-2} \mathrm{d} v \mathrm{d} r \\
&\quad\ + \int_0^{\eta} \int_r^{t_{j-1}} (\eta-r)^{\alpha-1} \big( (t_{j-1}-v)^{\alpha - 1} - (\tau-v)^{\alpha - 1} \big) (v-r)^{2H-2} \mathrm{d} v \mathrm{d} r.
\end{align*}
By \eqref{eq.Betafun} and \eqref{eq.eta1}, 
\begin{align*}
\mathcal T_1\leq C \int_0^{\eta} (\eta-v)^{\alpha+2H-2} \big( (t_{j-1}-v)^{\alpha - 1} - (\tau-v)^{\alpha - 1} \big) \mathrm{d} v \leq C h^{\alpha \wedge 2(\alpha+H-1)},
\end{align*}
and using a similar argument of \eqref{eq.Beta} gives
\begin{align*}
 \mathcal T_2 \leq&~ C \int_0^{\eta} (\eta-r)^{\alpha-1}\Big( (t_{j-1}-r)^{\alpha+2H-2} - (\tau-r)^{\alpha+2H-2} \\
 &\quad+\int_{t_{j-1}}^{\tau} (\tau-v)^{\alpha - 1} (v-r)^{2H-2} \mathrm{d} v \Big) \mathrm{d} r,
\end{align*}
which can be further bounded by $C\mathcal{R}_{H, \alpha} (h)$, in view of \eqref{eq.lem3.1} and \eqref{eq.eta2}. Thus, \eqref{eq.lem3.2} follows.

When $\eta \in (t_{j-1},T]$, the left hand side of \eqref{eq.lem3.2} can be bounded by
\begin{align*}
\mathcal T'_1 + \mathcal T'_2 &:= \int_0^{t_{j-1}} \int_0^{t_{j-1}} (t_{j-1}-r)^{\alpha-1} \big( (t_{j-1}-v)^{\alpha - 1} - (\tau-v)^{\alpha - 1} \big) |v-r|^{2H-2} \mathrm{d} v \mathrm{d} r \\
&\quad\ + \int_{t_{j-1}}^{\eta} \int_0^{t_{j-1}} (\eta-r)^{\alpha-1} \big( (t_{j-1}-v)^{\alpha - 1} - (\tau-v)^{\alpha - 1} \big) (r-v)^{2H-2} \mathrm{d} v \mathrm{d} r,
\end{align*}
in which $\mathcal T'_1 \leq C h^{\alpha \wedge 2(\alpha+H-1)}$ by the obtained result with $\eta = t_{j-1}$. Besides, 
\begin{align*}
\mathcal T'_2 &\leq C \int_0^{t_{j-1}} (\eta-v)^{\alpha+2H-2} \big( (t_{j-1}-v)^{\alpha - 1} - (\tau-v)^{\alpha - 1} \big) \mathrm{d} v \\
&\leq
\begin{cases}
Ch^{\alpha}, & \mbox{if } \alpha \in [2-2H,1),\\
Ch^{2(\alpha+H-1)}, & \mbox{if } \alpha \in (1-H,2-2H),\\
\end{cases}
\end{align*}
in which \eqref{chi} with $\theta=\alpha-1 \in (-1,0)$ is used to deal with the case $\alpha \in [2-2H,1)$, while the inequality
\begin{align*}
&\quad\ \int_0^{t_{j-1}} (\eta-v)^{\alpha+2H-2} \big( (t_{j-1}-v)^{\alpha - 1} - (\tau-v)^{\alpha - 1} \big) \mathrm{d} v\\
&\leq C \int_0^{t_{j-1}} (t_{j-1}-v)^{\alpha+2H-2} \big( (t_{j-1}-v)^{\alpha - 1} - (\tau-v)^{\alpha - 1} \big) \mathrm{d} v
\end{align*}
and \eqref{eq.eta1} are used to handle the case $\alpha \in (1-H,2-2H)$. Hence, \eqref{eq.lem3.2} holds for $\eta\in(t_{j-1},T]$. The proof is complete.
\end{proof}

\section{Proofs} \label{sec.Proof}

In this section, we provide the detailed proof of Theorem \ref{thm.mainEM}.

\subsection{Proof of Theorem \ref{thm.mainEM}}

We denote $\hat{s} := t_{j-1}$ for $s \in (t_{j-1}, t_j]$ with $j=1,2,\cdots,N$. By introducing
\begin{align*}
R_n := \sum_{j=1}^n \int_{t_{j-1}}^{t_j} (t_n-s)^{\alpha-1} \big( b(x(s)) - b(x(\hat{s})) \big) \mathrm{d}s, \quad n \in \{1, 2, \cdots, N \}
\end{align*}
and according to \eqref{eq.GLE} and \eqref{eq.EM}, the strong error of the Euler method satisfies
\begin{align*}
\| x_n - x(t_n) \|
&\leq \frac{1}{\Gamma(\alpha)} \Big\| \sum_{j=1}^n \int_{t_{j-1}}^{t_j} (t_n-s)^{\alpha-1} \big( b(x_{j-1}) - b(x(\hat{s})) \big)\mathrm{d}s \Big\| \\
&\quad\ + \frac{1}{\Gamma(\alpha)} \Big\| \sum_{j=1}^n \int_{t_{j-1}}^{t_j} (t_n-s)^{\alpha-1} \big( b(x(\hat{s})) - b(x(s)) \big)\mathrm{d}s \Big\| \\
&\leq C \sum_{j=1}^n \int_{t_{j-1}}^{t_j} (t_n-s)^{\alpha-1} \| x_{j-1} - x(\hat{s}) \| \mathrm{d}s + \frac{1}{\Gamma(\alpha)} \| R_n \|,
\end{align*}
for any $n \in \{1, 2, \cdots, N \}$, where the last step used the assumption $b \in C_b^1$. Then, applying the singular Gr\"{o}nwall inequality yields
\begin{align} \label{eq.xn-xtoRn}
\| x_n - x(t_n) \| \leq C \| R_n \|.
\end{align}

In order to estimate $\| R_n \|$, the GLE \eqref{eq.GLE} is reformulated as
\begin{align} \label{eq.splitGLE}
x(t) = \zeta(t) + G(t), \qquad \text{ where } \zeta(t):= x_0 + \frac{1}{\Gamma(\alpha)}\int_0^{t}(t-s)^{\alpha-1}b(x(s)) \mathrm{d}s.
\end{align}
The mean value theorem implies $b(x(s)) - b(x(\hat{s})) = \big(x(s) - x(\hat{s})\big) \int_0^1 b'\big( \xi_{s}^{\theta} \big) \mathrm d \theta$ with 
\begin{align} \label{eq.MeanThm}
\xi_{s}^{\theta} := x(\hat{s}) + \theta \big(x(s) - x(\hat{s})\big).
\end{align}
This along with \eqref{eq.splitGLE} gives
\begin{align}\label{eq.Rn}
\|R_n\| &\leq \Big\| \sum_{j=1}^{n} \int_{t_{j-1}}^{t_j} (t_n - s)^{\alpha-1} \big( \zeta(s) - \zeta(\hat{s}) \big) \int_0^1 b'\big( \xi_{s}^{\theta} \big) \mathrm d \theta \mathrm{d}s \Big\| \nonumber \\
&\quad + \Big\| \sum_{j=1}^{n} \int_{t_{j-1}}^{t_j} (t_n - s)^{\alpha-1} \big( G(s) - G(\hat{s}) \big) \int_0^1 b'\big( \xi_{s}^{\theta} \big) \mathrm d \theta \mathrm{d}s \Big\| =: \|R_{n, 1}\| + \|R_{n, 2}\|.
\end{align}
Since $b \in C_b^1$ and $\zeta$ is $\alpha$-H\"{o}lder continuous in $L^2(\Omega;\mathbb{R})$ (see \cite[Page 456]{FangLi2020}),\begin{align}
\|R_{n, 1}\|
&\leq C\sum_{j=1}^{n} \int_{t_{j-1}}^{t_j} (t_n - s)^{\alpha-1} \big\| \zeta(s) - \zeta(\hat{s}) \big\|\mathrm{d}s \nonumber \\
&\leq Ch^\alpha\sum_{j=1}^{n} \int_{t_{j-1}}^{t_j} (t_n - s)^{\alpha-1} \mathrm{d}s\leq Ch^\alpha. \label{Rn,1}
\end{align}
To estimate $\| R_{n, 2} \|$, we denote
\begin{align*}
I_{n, 2}^j := \int_0^1 \int_{t_{j-1}}^{t_j} (t_n - s)^{\alpha-1} \big( G(s) - G(\hat{s}) \big) b'\big( \xi_{s}^{\theta} \big) \mathrm{d}s \mathrm d \theta ,
\end{align*}
and make the decomposition
\begin{align*}
\|R_{n, 2}\|^2 &=\sum_{j=1}^{n} \|I_{n, 2}^j\|^2 + 2 \sum_{1\leq i <j \leq n} \langle I_{n, 2}^i, I_{n, 2}^j \rangle.
\end{align*}
For the first term of the right hand side, using the fact that $G(\cdot)$ is $(H+\alpha-1)$-H\"{o}lder continuous in $L^2(\Omega;\mathbb{R})$ (see \cite[Proposition 1]{LiLiu2017}) yields
\begin{align*}
&\quad\ \sum_{j=1}^{n} \|I_{n, 2}^j\|^2 \leq C \sum_{j=1}^{n} \left( \int_{t_{j-1}}^{t_j} (t_n - s)^{\alpha-1} \| G(s) - G(\hat{s}) \| \mathrm{d}s \right)^2 \\
&\leq C h^{2(H + \alpha - 1)} \sum_{j=1}^{n} \left( \int_{t_{j-1}}^{t_j} (t_n - s)^{\alpha-1} \mathrm{d}s \right)^2
\leq C h^{2(H + 2\alpha - 1)}\bigg(1+\sum_{j=1}^{n-1}(n-j)^{2(\alpha-1)}\bigg)\\
&\leq\begin{cases}
C h^{2H + 4\alpha - 2}, & \mbox{if } 1-H<\alpha<\frac{1}{2}, \\
C \big( |\ln h| \vee \ln T \big) h^{2H}, & \mbox{if } \alpha=\frac{1}{2}, \\
C h^{2H + 2\alpha -1}, & \mbox{if } \frac{1}{2}<\alpha<1,
\end{cases}
\end{align*}
which can be further bounded by $C\mathcal{R}_{H, \alpha}^2(h)$. Thus,
\begin{align}\label{eq.Rn2}
\|R_{n, 2}\|^2 \leq C \mathcal{R}_{H, \alpha}^2(h)+2 \sum_{1\leq i <j \leq n} \langle I_{n, 2}^i, I_{n, 2}^j \rangle.
\end{align}
Notice that $\langle I_{n, 2}^i, I_{n, 2}^j \rangle$ is equal to
\begin{align*}
\int_0^1 \int_0^1 \int_{t_{i-1}}^{t_i} \int_{t_{j-1}}^{t_j} (t_n - s)^{\alpha-1} (t_n - \tau)^{\alpha-1} \mathbb{E}\big[ b'\big( \xi_{s}^{\theta} \big) b'\big( \xi_{\tau}^{\lambda} \big) \big(G(s) - G(\hat{s})\big) \big(G(\tau) - G(\hat{\tau})\big) \big] \mathrm{d}\tau \mathrm{d}s \mathrm d \theta \mathrm d \lambda.
\end{align*}

Our main efforts for the proof of Theorem \ref{thm.mainEM} lies in the domination of the expectation in the integrand of the above formula, as the subsequent Proposition \ref{lem.covGt} shows. Recall that $G(\cdot)$ is defined by \eqref{eq.Gt}.

\begin{proposition} \label{lem.covGt}
For $ s \in (t_{i-1}, t_i]$, $\tau \in (t_{j-1}, t_j]$ with $1 \leq i < j \leq N$, and $\theta,\lambda\in(0,1)$, let $ \xi_{s}^{\theta}$, $\xi_{\tau}^{\lambda}$ be given by \eqref{eq.MeanThm}. Under the assumptions of Theorem \ref{thm.mainEM},
\begin{align*}
\mathbb{E}\Big[ b'\big( \xi_{s}^{\theta} \big) b'\big( \xi_{\tau}^{\lambda} \big) \big(G(s) - G(\hat{s})\big) \big(G(\tau) - G(\hat{\tau})\big) \Big]
\leq C \mathcal{R}_{H, \alpha}^2(h) (\tau-s)^{2H-2},
\end{align*}
where the constant $C=C(\alpha,H,T,\sigma)$ is independent of $\theta,\lambda,s,\tau,i,j$.
\end{proposition}
\noindent The proof of Proposition \ref{lem.covGt} is deferred to Subsection \ref{sec.PropositioncovGt} without
interrupting the flow of the proof of Theorem \ref{thm.mainEM}.

With Proposition \ref{lem.covGt} in mind, one can derive
\begin{align}\label{In,2}\notag
 \sum_{1\leq i <j \leq n} \langle I_{n, 2}^i, I_{n, 2}^j \rangle\notag
&\leq C \mathcal{R}_{H, \alpha}^2(h) \int_0^{t_{n-1}} \int_s^{t_n} (t_n -s)^{\alpha - 1} (t_n -\tau)^{\alpha - 1} (\tau - s)^{2H-2} \mathrm{d} \tau \mathrm{d} s \\\notag
&= C \mathcal{R}_{H, \alpha}^2(h) \int_h^{t_n} \int_0^u u^{\alpha-1} v^{\alpha-1} (u-v)^{2H-2} \mathrm{d}v \mathrm{d}u \\
&\leq C \mathcal{R}_{H, \alpha}^2(h) \int_h^{t_n} u^{2\alpha + 2H -3}\mathrm{d}u \leq C \mathcal{R}_{H, \alpha}^2(h),
\end{align}
where the last step used the fact $2\alpha + 2H -3 > -1$ since $\alpha \in (1-H,1)$. Substituting \eqref{In,2} into \eqref{eq.Rn2} gives $\|R_{n,2}\| \leq C \mathcal{R}_{H, \alpha} (h)$, which along with \eqref{eq.Rn} and \eqref{Rn,1} reveals
\begin{align*}
\|R_n\| \leq C\mathcal{R}_{H, \alpha} (h).
\end{align*}
Finally, recalling the estimate \eqref{eq.xn-xtoRn} completes the proof of Theorem \ref{thm.mainEM}.
\qed

\begin{remark}
We would like to mention that by the same strategy as in the proof of \cite[Page 5]{DaiXiao2021}, the estimate \eqref{Rn,1} for $\|R_{n, 1}\|$ could be improved to be
\begin{align*}
\|R_{n, 1}\|^2 \leq
\begin{cases}
C h^{(2H+4\alpha-2) \wedge (3\alpha)}, & \mbox{if } 1-H < \alpha < 1/2, \\
C \max\big\{ h^{2H}, |\ln h| h^{3/2} \big\}, & \mbox{if } \alpha = 1/2, \\
C h^{(2H+4\alpha-2) \wedge (\alpha+1)}, & \mbox{if } 1/2 < \alpha < 1-H/2, \\
C h^{\alpha+1}, & \mbox{if } 1-H/2 \leq \alpha < 1,
\end{cases}
\end{align*}
which means that the error terms $\|R_{n, 1}\|^2$ and $\sum_{j=1}^{n} \|I_{n, 2}^j\|^2$ share the same rates with those of the linear case. Hence, the rest error term $2 \sum_{1\leq i <j \leq n} \langle I_{n, 2}^i, I_{n, 2}^j \rangle$ dominates the error of the Euler method \eqref{eq.EM} for the overdamped GLE \eqref{eq.GLE} in the nonlinear case.
\end{remark}

\subsection{Proof of Proposition \ref{lem.covGt}} \label{sec.PropositioncovGt}

For all $0 < s < t \leq T$, the increment of the diffusion term can be rewritten as
\begin{align*}
G(t) - G(s) &= \frac{\sigma}{\Gamma(\alpha)} \Big( \int_0^s (t-u)^{\alpha - 1} - (s-u)^{\alpha - 1} \mathrm{d}W_H(u) + \int_s^t (t-u)^{\alpha-1} \mathrm{d}W_H(u) \Big) \\
&= \frac{\sigma}{\Gamma(\alpha)} \delta \Big( ((t-\cdot)^{\alpha - 1} - (s-\cdot)^{\alpha - 1}) \mathbf 1_{(0,s)}(\cdot) + (t-\cdot)^{\alpha - 1} \mathbf 1_{(s,t)}(\cdot) \Big),
\end{align*}
where $\delta$ is the Skorohod integral introduced in Subsection \ref{sec.MalliavinfBm}. For the simplicity of notations, for $s \in (0,T]$, we denote 
\begin{equation*} 
\begin{split}
& \widehat{A} (\cdot;s) := \big( (\hat{s} - \cdot)^{\alpha - 1} - (s-\cdot)^{\alpha - 1} \big) \mathbf 1_{(0,\hat{s})}(\cdot), \quad \widetilde{A} (\cdot;s) : = (s-\cdot)^{\alpha - 1} \mathbf 1_{(\hat{s},s)}(\cdot), \\
& A (\cdot;s): = \widehat{A} (\cdot;s) + \widetilde{A} (\cdot;s), \qquad\qquad\qquad\quad\quad\ \, U (\cdot;s) := -\widehat{A} (\cdot;s) + \widetilde{A} (\cdot;s).
\end{split}
\end{equation*}
It follows from the dual formula \cite[Eq.\ (25)]{KloedenNeuenkirch2011} that
\begin{align}\label {eq.main2term} 
&\quad\ \mathbb{E}\Big[ b'( \xi_{s}^{\theta} ) b'( \xi_{\tau}^{\lambda} ) \big(G(s) - G(\hat{s})\big) \big(G(\tau) - G(\hat{\tau})\big) \Big] = \frac{\sigma^2}{\Gamma^2(\alpha)} \mathbb{E}\Big[ b'( \xi_{s}^{\theta} ) b'( \xi_{\tau}^{\lambda} ) \delta\big( U (\cdot;s) \big) \delta\big( U (\cdot;\tau) \big) \Big] \notag\\
&= \frac{\sigma^2}{\Gamma^2(\alpha)} \mathbb{E}\big[ b'( \xi_{s}^{\theta} ) b'( \xi_{\tau}^{\lambda} ) \langle U (\cdot;s), U (\cdot;\tau) \rangle_{\mathcal{H}} \big] + \frac{\sigma^2}{\Gamma^2(\alpha)} \mathbb{E}\big[ \langle D \langle D \big(b'( \xi_{s}^{\theta} ) b'( \xi_{\tau}^{\lambda} )\big), U (\cdot;s) \rangle_{\mathcal{H}}, U (\cdot;\tau) \rangle_{\mathcal{H}} \big] \notag \\
&\leq C \int_{[0,T]^2} \| b'( \xi_{s}^{\theta} ) b'( \xi_{\tau}^{\lambda} ) \|_1 A(u;s) A(v;\tau) |u-v|^{2H-2} \mathrm{d} u \mathrm{d} v \notag \\
&\quad + C \int_{[0,T]^4} \| D_{r_2} D_{r_1} \big( b'( \xi_{s}^{\theta} ) b'( \xi_{\tau}^{\lambda} ) \big) \|_1 A(u;s) A(v;\tau) |r_1 - u|^{2H-2} |r_2 - v|^{2H-2} \mathrm{d} r_1 \mathrm{d} u \mathrm{d} r_2 \mathrm{d}v. 
\end{align}
Applying the product rule and chain rule of the Malliavin derivative obtains
\begin{align}\label{eq.DF}
D_{r_1} \big( b'( \xi_{s}^{\theta} ) b'( \xi_{\tau}^{\lambda} ) \big) = b''(\xi_{s}^{\theta}) D_{r_1} \xi_{s}^{\theta} b'( \xi_{\tau}^{\lambda} ) + b'( \xi_{s}^{\theta} ) b''(\xi_{\tau}^{\lambda}) D_{r_1} \xi_{\tau}^{\lambda},
\end{align}
and
\begin{align}
D_{r_2} D_{r_1} \big( b'( \xi_{s}^{\theta} ) b'( \xi_{\tau}^{\lambda} ) \big)
&= b'''(\xi_{s}^{\theta}) D_{r_2} \xi_{s}^{\theta} D_{r_1} \xi_{s}^{\theta} b'(\xi_{\tau}^{\lambda}) + b''(\xi_{s}^{\theta}) D_{r_2}D_{r_1} \xi_{s}^{\theta} b'(\xi_{\tau}^{\lambda}) \notag \\
&\quad + b''(\xi_{s}^{\theta}) D_{r_1} \xi_{s}^{\theta} b''(\xi_{\tau}^{\lambda}) D_{r_2} \xi_{\tau}^{\lambda} + b''(\xi_{s}^{\theta}) D_{r_2} \xi_{s}^{\theta} b''(\xi_{\tau}^{\lambda}) D_{r_1} \xi_{\tau}^{\lambda}\notag\\
&\quad + b'( \xi_{s}^{\theta}) b'''(\xi_{\tau}^{\lambda}) D_{r_2} \xi_{\tau}^{\lambda} D_{r_1} \xi_{\tau}^{\lambda} + b'( \xi_{s}^{\theta}) b''(\xi_{\tau}^{\lambda}) D_{r_2} D_{r_1} \xi_{\tau}^{\lambda}. \label{eq.D2F}
\end{align}

Invoking \eqref{eq.DF} and \eqref{eq.D2F}, it follows from $b \in C_{b,p}^{1,3}$, H\"older's inequality and \eqref{eq.XtLq} that
\begin{align*}
\| D_{r_2} D_{r_1} \big( b'( \xi_{s}^{\theta} ) b'( \xi_{\tau}^{\lambda} ) \big) \|_1 &\leq C \big( \|D_{r_2}D_{r_1} \xi_{s}^{\theta}\| + \| D_{r_2} D_{r_1} \xi_{\tau}^{\lambda}\| \big) \\
&\quad + C \big( \|D_{r_1} \xi_{s}^{\theta}\|_4 + \|D_{r_1} \xi_{\tau}^{\lambda}\|_4 \big) \big( \|D_{r_2} \xi_{s}^{\theta}\|_4 + \|D_{r_2} \xi_{\tau}^{\lambda}\|_4 \big).
\end{align*}
Note that $\| b'( \xi_{s}^{\theta} ) b'( \xi_{\tau}^{\lambda} ) \|_1$ is bounded since $b \in C_b^1$. Thus, by \eqref {eq.main2term},
\begin{align}\label{eq.JKL}
\mathbb{E}\Big[ b'( \xi_{s}^{\theta} ) b'( \xi_{\tau}^{\lambda} ) \big(G(s) - G(\hat{s})\big) \big(G(\tau) - G(\hat{\tau})\big) \Big]
\leq C \big\{ \mathcal{J} + \mathcal{K} + \mathcal{L} \big\},
\end{align}
where
\begin{align}\label{eq.J}
\mathcal{J} &:=\int_{[0,T]^2} A(u;s) A(v;\tau) |u-v|^{2H-2} \mathrm{d} u \mathrm{d} v, \\\label{eq.K}\notag
\mathcal{K} &:= \int_{[0,T]^2} \big( \|D_{r_1} \xi_{s}^{\theta}\|_4 + \|D_{r_1} \xi_{\tau}^{\lambda}\|_4 \big) A(u;s) |r_1 - u|^{2H-2} \mathrm{d} r_1 \mathrm{d} u \\
&\quad\ \times \int_{[0,T]^2} \big( \|D_{r_2} \xi_{s}^{\theta}\|_4 + \|D_{r_2} \xi_{\tau}^{\lambda}\|_4 \big) A(v;\tau) |r_2 - v|^{2H-2} \mathrm{d} r_2 \mathrm{d}v, \\\label{eq.L}\notag
\mathcal{L} &:= \int_{[0,T]^4} \big( \|D_{r_2}D_{r_1} \xi_{s}^{\theta}\| + \|D_{r_2}D_{r_1} \xi_{\tau}^{\lambda}\| \big) \\
&\qquad\qquad\ \ \times A(u;s) A(v;\tau) |r_1 - u|^{2H-2} |r_2 - v|^{2H-2} \mathrm{d} r_1 \mathrm{d} u \mathrm{d} r_2 \mathrm{d}v.
\end{align}
Finally, the proof of Proposition \ref{lem.covGt} is completed by combining Lemmas \ref{propJ} and \ref{propKL} with \eqref{eq.JKL}.
\qed

The rest of the section is devoted to proving Lemmas \ref{propJ} and \ref{propKL}.

\begin{lemma}\label{propJ}
For $\mathcal J$ given by \eqref{eq.J}, there exists some constant $C$ such that
\begin{align*}
 \mathcal{J}\leq Ch^{2\alpha}(\tau-s)^{2H-2},
\end{align*}
for any $s\in(t_{i-1},t_i]$ and $\tau\in(t_{j-1},t_j]$ with $1\leq i<j\leq N$.
\end{lemma}

\begin{proof}
According to $A=\widehat{A} + \widetilde{A}$ and \eqref{eq.J}, we have $\mathcal{J}= \mathcal{J}_1 + \mathcal{J}_2 + \mathcal{J}_3 + \mathcal{J}_4$ with
\begin{gather*}
\mathcal{J}_1: = \int_{[0,T]^2} \widehat{A}(u;s) \widehat{A}(v;\tau) |u-v|^{2H-2} \mathrm{d} u \mathrm{d} v, \ \ 
\mathcal{J}_2:= \int_{[0,T]^2} \widehat{A}(u;s) \widetilde{A}(v;\tau) |u-v|^{2H-2} \mathrm{d} u \mathrm{d} v, \\
\mathcal{J}_3:= \int_{[0,T]^2} \widetilde{A}(u;s) \widehat{A}(v;\tau) |u-v|^{2H-2} \mathrm{d} u \mathrm{d} v, \ \ 
\mathcal{J}_4:= \int_{[0,T]^2} \widetilde{A}(u;s) \widetilde{A}(v;\tau) |u-v|^{2H-2} \mathrm{d} u \mathrm{d} v.
\end{gather*}

\textbf{Estimate of $\mathcal{J}_1$.} We firstly decompose $\mathcal{J}_1$ into
\begin{align*}
\mathcal{J}_1 &= \int_0^{t_{i-1}} \int_0^{t_{j-1}} \big( (t_{i-1}-u)^{\alpha - 1} - (s-u)^{\alpha - 1} \big) \big( (t_{j-1}-v)^{\alpha - 1} - (\tau-v)^{\alpha - 1} \big) |v-u|^{2H-2} \mathrm{d}v \mathrm{d}u \\
&= \int_0^{t_{i-1}} \int_0^{t_{i-1}}\big( (t_{i-1}-u)^{\alpha - 1} - (s-u)^{\alpha - 1} \big) \big( (t_{j-1}-v)^{\alpha - 1} - (\tau-v)^{\alpha - 1} \big) |v-u|^{2H-2} \mathrm{d}v \mathrm{d}u\\
&\quad\ + \int_0^{t_{i-1}} \int_{t_{i-1}}^{t_{j-1}}\big( (t_{i-1}-u)^{\alpha - 1} - (s-u)^{\alpha - 1} \big) \big( (t_{j-1}-v)^{\alpha - 1} - (\tau-v)^{\alpha - 1} \big) |v-u|^{2H-2} \mathrm{d}v \mathrm{d}u\\
&=: \mathcal{J}_{11}+ \mathcal{J}_{12}.
\end{align*}
In order to estimate $\mathcal{J}_{11}$, we note that for $v\in[0,t_{i-1}]$, one has $\tau-s \leq 2(t_{j-1}-v)$, which implies
\begin{align}\label{eq.mean}
(t_{j-1}-v)^{\alpha - 1} - (\tau-v)^{\alpha - 1} \leq h(t_{j-1}-v)^{\alpha - 2}\leq Ch(t_{j-1}-v)^{\alpha - 2H} (\tau-s)^{2H-2}.
\end{align}
Besides,
\begin{align} \label{eq.J11-1}
&\quad\ \int_0^{t_{i-1}} \int_0^{t_{i-1}} \big( (t_{i-1}-u)^{\alpha - 1} - (s-u)^{\alpha - 1} \big)(t_{j-1}-v)^{\alpha - 2H} |v-u|^{2H-2} \mathrm{d}v \mathrm{d}u\\\notag
&\leq \int_0^{t_{i-1}} \int_{v}^{t_{i-1}} \big( (t_{i-1}-u)^{\alpha - 1} - (s-u)^{\alpha - 1} \big) (t_{j-1}-v)^{\alpha - 2H} (u-v)^{2H-2} \mathrm{d}u\mathrm{d}v\\\notag
&\quad + h^{\alpha-1}\int_0^{t_{i-1}} \int_u^{t_{i-1}} \big( (t_{i-1}-u)^{\alpha - 1} - (s-u)^{\alpha - 1} \big) (t_{j-1}-v)^{1- 2H} (v-u)^{2H-2} \mathrm{d}v\mathrm{d}u\\
&\leq C h^{2\alpha-1}
+Ch^{\alpha-1} \int_0^{t_{i-1}} (t_{i-1}-u)^{\alpha - 1} - (s-u)^{\alpha - 1}\mathrm{d}u\leq C h^{2\alpha-1}, \notag
\end{align}
in which $(t_{j-1}-v)^{\alpha-1}\leq h^{\alpha-1}$ for $v\leq t_{i-1}<t_{j-1}$ was used in the first inequality, \eqref{eq.lem5.1-1}, $t_{i-1}<t_{j-1}$, and \eqref{eq.Betafun} were used in the second inequality, and \eqref{chi} with $\theta=\alpha-1 \in (-1,0)$ was used in the last inequality. Then, the combination of \eqref{eq.mean} and \eqref{eq.J11-1} reveals
\begin{align}\label{eq.J11}
\mathcal{J}_{11} \leq Ch^{2\alpha}(\tau-s)^{2H-2}.
\end{align}
By the virtue of Lemma \ref{eq.ijts} with $k=i$ and $l=j$, as well as \eqref{chi} with $\theta=\alpha-1$,
\begin{align}\label{eq.J12}
\mathcal{J}_{12} &\leq Ch^{\alpha}(\tau-s)^{2H-2} \int_0^{t_{i-1}} (t_{i-1}-u)^{\alpha - 1} - (s-u)^{\alpha - 1} \mathrm{d}u \leq C h^{2\alpha}(\tau-s)^{2H-2}.
\end{align}
Collecting \eqref{eq.J11} and \eqref{eq.J12}, one obtains $\mathcal{J}_{1}=\mathcal{J}_{11}+\mathcal{J}_{12}\leq C h^{2\alpha}(\tau-s)^{2H-2}$.

\textbf{Estimate of $\mathcal{J}_2$.} For $v \in [t_{j-1},\tau]$ and $u \in [0,t_{i-1}]$, $\tau -s = \tau -v + v - s \leq h + v - u \leq 2(v-u)$, which along with \eqref{chi} with $\theta=\alpha-1$ implies
\begin{align*}
\mathcal{J}_2 &= \int_0^{t_{i-1}} \int_{t_{j-1}}^{\tau} \big( (t_{i-1}-u)^{\alpha - 1} - (s-u)^{\alpha - 1} \big) (\tau-v)^{\alpha - 1} (v-u)^{2H-2} \mathrm{d}v \mathrm{d}u \\
&\leq C \int_0^{t_{i-1}} \int_{t_{j-1}}^{\tau} \big( (t_{i-1}-u)^{\alpha - 1} - (s-u)^{\alpha - 1} \big) (\tau-v)^{\alpha - 1} \mathrm{d}v \mathrm{d}u (\tau - s)^{2H-2} \\
&\leq C h^{2\alpha} (\tau - s)^{2H-2}.
\end{align*}

\textbf{Estimate of $\mathcal{J}_3$.} To facilitate the estimation of $\mathcal{J}_3$, we note that
\begin{align*}
\mathcal{J}_3 &= \int_{t_{i-1}}^s \int_0^{t_{j-1}} (s-u)^{\alpha - 1} \big( (t_{j-1}-v)^{\alpha - 1} - (\tau-v)^{\alpha - 1} \big) |v-u|^{2H-2} \mathrm{d}v \mathrm{d}u \\
&= \int_{t_{i-1}}^s \int_0^{t_{i-1}}
(s-u)^{\alpha - 1} \big( (t_{j-1}-v)^{\alpha - 1} - (\tau-v)^{\alpha - 1} \big) (u-v)^{2H-2} \mathrm{d}v \mathrm{d}u\\
&\quad\ + \int_{t_{i-1}}^s \int_{t_{i-1}}^{t_{j-1}} (s-u)^{\alpha - 1} \big( (t_{j-1}-v)^{\alpha - 1} - (\tau-v)^{\alpha - 1} \big) |v-u|^{2H-2} \mathrm{d}v \mathrm{d}u\\
& =: \mathcal{J}_{3,1} + \mathcal{J}_{3,2}.
\end{align*}
It follows from \eqref{eq.mean} and \eqref{eq.lem5.1-2} that
\begin{align*}
\mathcal{J}_{3,1}
&\leq Ch \int_{t_{i-1}}^s \int_0^{t_{i-1}} (s-u)^{\alpha - 1} (u-v)^{2H-2} (t_{j-1}-v)^{\alpha-2H} \mathrm{d}v \mathrm{d}u (\tau-s)^{2H-2} \\
&\leq Ch^{2\alpha} (\tau-s)^{2H-2}.
\end{align*}
The estimate of $\mathcal{J}_{3,2}$ can be divided into two cases: $i < j \leq i+2$ and $i+3 \leq j \leq N$. If $i < j \leq i+2$, then similar to \eqref{eq.3intgrand}, we have
\begin{align} \label{eq.J31C1}
\mathcal{J}_{3,2} &\leq \int_{t_{i-1}}^s \int_{t_{i-1}}^{t_{j-1}} (s-u)^{\alpha - 1} (t_{j-1}-v)^{\alpha - 1} |v-u|^{2H-2} \mathrm{d}v \mathrm{d}u \\
&\leq C h^{2(\alpha+H-1)} \leq C h^{2\alpha} (\tau-s)^{2H-2}, \notag
\end{align}
where the last step used $\tau - s\leq t_{j}-t_{i-1}\leq 3h$. If $i+3 \leq j \leq N$, we
apply \eqref{eq.J31C1} to obtain
\begin{align*}
&\quad\ \int_{t_{i-1}}^s \int_{t_{i-1}}^{t_i} (s-u)^{\alpha - 1} \big( (t_{j-1}-v)^{\alpha - 1} - (\tau-v)^{\alpha - 1} \big) |v-u|^{2H-2} \mathrm{d}v \mathrm{d}u\\
&\leq \int_{t_{i-1}}^s \int_{t_{i-1}}^{t_i} (s-u)^{\alpha - 1} (t_{i}-v)^{\alpha - 1} |v-u|^{2H-2} \mathrm{d}v \mathrm{d}u \leq Ch^{2\alpha} (\tau-s)^{2H-2},
\end{align*}
and utilize Lemma \ref{eq.ijts} with $k=i+1$ and $l=j$ to get
\begin{align*}
&\quad \int_{t_{i-1}}^s \int_{t_i}^{t_{j-1}} (s-u)^{\alpha - 1} \big( (t_{j-1}-v)^{\alpha - 1} - (\tau-v)^{\alpha - 1} \big) |v-u|^{2H-2} \mathrm{d}v \mathrm{d}u\\
 &\leq Ch^{\alpha}(\tau -s)^{2H-2}\int_{t_{i-1}}^s(s-u)^{\alpha - 1} \mathrm{d}u \leq C h^{2\alpha} (\tau -s)^{2H-2}.
\end{align*}
These imply $\mathcal{J}_{32} \leq C h^{2\alpha} (\tau -s)^{2H-2}$ for the case $i+3 \leq j \leq N$. Hence, we have $\mathcal{J}_{3} \leq C h^{2\alpha} (\tau -s)^{2H-2}$.

\textbf{Estimate of $\mathcal{J}_4$.} It can be divided into two cases:\ $j = i + 1$ and $i+2 \leq j \leq N$. When $j = i + 1$,
\begin{align*}
\mathcal{J}_4 &= \int_{t_{i-1}}^s \int_{t_{i}}^{\tau} (s-u)^{\alpha - 1} (\tau-v)^{\alpha - 1} (v-u)^{2H-2} \mathrm{d}v \mathrm{d}u \\
&\leq \int_{t_{i-1}}^s \int_{s}^{\tau} (s-u)^{\alpha - 1} (\tau-v)^{\alpha - 1} (v-s)^{2H-2} \mathrm{d}v \mathrm{d}u \\
&\leq C (\tau-s)^{\alpha+2H-2} \int_{t_{i-1}}^s (s-u)^{\alpha - 1} \mathrm{d}u \leq C h^{2\alpha} (\tau-s)^{2H-2}.
\end{align*}
When $i+2 \leq j \leq N$, for $v \in [t_{j-1},\tau]$ and $u \in [t_{i-1},s]$, we have $$\tau-s \leq (j-i+1)h \leq 3(j-i-1)h \leq 3(v-u),$$ and thus,
\begin{align*}
\mathcal{J}_4 \leq C \int_{t_{i-1}}^s \int_{t_{j-1}}^{\tau} (s-u)^{\alpha - 1} (\tau-v)^{\alpha - 1} \mathrm{d}v \mathrm{d}u (\tau-s)^{2H-2} \leq C h^{2\alpha} (\tau-s)^{2H-2}.
\end{align*}

Finally, gathering the above estimates of $\mathcal{J}_1,\, \mathcal{J}_2,\, \mathcal{J}_3 $ and $\mathcal{J}_4$ together completes the proof.
\end{proof}

\begin{lemma}\label{propKL}
For $\mathcal K$ and $\mathcal L$ respectively given by \eqref{eq.K} and \eqref{eq.L},
there exists some constant $C$ such that
\begin{align*}
\mathcal K+\mathcal L \leq C \mathcal{R}_{H, \alpha}^2(h),
\end{align*}
for any $s\in(t_{i-1},t_i]$,~$\tau\in(t_{j-1},t_j]$ with $1\leq i<j\leq N$, and $\theta,\, \lambda \in (0,1)$.
\end{lemma}

\begin{proof}
According to Lemma \ref{lem.impLemUse}, one has that for any $\tau \in (t_{j-1},t_j]$ with $1 \leq j \leq N$, 
\begin{align} \label{eq.KHbound}
\sup_{\mu \in [0,T]} \int_0^{\mu} \int_0^T (\mu - r)^{\alpha-1} A(v;\tau) |r - v|^{2H-2} \mathrm{d} v \mathrm{d} r \leq C \mathcal{R}_{H, \alpha}(h).
\end{align}
Then, taking \eqref{Drxt} and \eqref{eq.KHbound} into account shows
\begin{align*}
\mathcal{K} &\le C \sup_{\mu_1 \in [0,T]} \int_0^{\mu_1} \int_0^T (\mu_1 - r_1)^{\alpha-1} A(u;s) |r_1 - u|^{2H-2} \mathrm{d} u \mathrm{d} r_1 \\
&\quad \times \sup_{\mu_2 \in [0,T]} \int_0^{\mu_2} \int_0^T (\mu_2 - r_2)^{\alpha-1} A(v;\tau) |r_2 - v|^{2H-2} \mathrm{d} v \mathrm{d} r_2 \leq C \mathcal{R}_{H, \alpha}^2(h).
\end{align*}
Besides, \eqref{Drxt-2} and \eqref{eq.KHbound} display
\begin{align*}
\mathcal{L}
&\leq C \sup_{\mu \in [0,T]} \int_0^{\mu} \int_{[0,T]^4} (\mu-\eta)^{\alpha-1} (\eta - r_1)^{\alpha - 1} (\eta - r_2)^{\alpha - 1} \mathbf 1_{[0,\eta)} (r_1 \vee r_2) \\
&\qquad\qquad\qquad\qquad\quad\ \times A(u;s) A(v;\tau) |r_1 - u|^{2H-2} |r_2 - v|^{2H-2} \mathrm{d} r_1 \mathrm{d} u \mathrm{d} r_2 \mathrm{d} v \mathrm{d} \eta \\
&\leq C \mathcal{R}_{H, \alpha}^2(h) \sup_{\mu \in [0,T]} \int_0^{\mu} (\mu-\eta)^{\alpha-1} \mathrm{d} \eta \leq C \mathcal{R}_{H, \alpha}^2(h).
\end{align*}
The proof is complete.
\end{proof}


\textbf{Acknowledgments.} The authors are grateful to Professor Lei Li for helpful discussions.

\bibliographystyle{plain}
\bibliography{references}

\begin{thebibliography}{10}

\bibitem{DaiXiao2021}
X.~Dai and A.~Xiao.
\newblock A note on {E}uler method for the overdamped generalized {L}angevin
  equation with fractional noise.
\newblock {\em Appl. Math. Lett.}, 111, 2021.

\bibitem{Dalang1999}
R.~C. Dalang.
\newblock Extending the martingale measure stochastic integral with
  applications to spatially homogeneous {S.P.D.E.'s}.
\newblock {\em Electron. J. Probab.}, 4:1--29, 1999.

\bibitem{DidierNguyen2020}
G.~Didier and H.~Nguyen.
\newblock Asymptotic analysis of the mean squared displacement under fractional
  memory kernels.
\newblock {\em SIAM J. Math. Anal.}, 52(4):3818--3842, 2020.

\bibitem{FangLi2020}
D.~Fang and L.~Li.
\newblock Numerical approximation and fast evaluation of the overdamped
  generalized {L}angevin equation with fractional noise.
\newblock {\em ESAIM Math. Model. Numer. Anal.}, 54(2):431--463, 2020.

\bibitem{Giles2008}
M.~B. Giles.
\newblock Multilevel {M}onte {C}arlo path simulation.
\newblock {\em Oper. Res.}, 56(3):607--617, 2008.

\bibitem{Giles2015}
M.~B. Giles.
\newblock Multilevel {M}onte {C}arlo methods.
\newblock {\em Acta Numer.}, 24:259--328, 2015.

\bibitem{HongHuang2020}
J.~Hong, C.~Huang, M.~Kamrani, and X.~Wang.
\newblock Optimal strong convergence rate of a backward {E}uler type scheme for
  the {C}ox--{I}ngersoll--{R}oss model driven by fractional {B}rownian motion.
\newblock {\em Stochastic Process. Appl.}, 130(5):2675--2692, 2020.

\bibitem{JiangZhang2017}
S.~Jiang, J.~Zhang, Q.~Zhang, and Z.~Zhang.
\newblock Fast evaluation of the {C}aputo fractional derivative and its
  applications to fractional diffusion equations.
\newblock {\em Commun. Comput. Phys.}, 21(3):650--678, 2017.

\bibitem{KloedenNeuenkirch2011}
P.~E. Kloeden, A.~Neuenkirch, and R.~Pavani.
\newblock Multilevel {M}onte {C}arlo for stochastic differential equations with
  additive fractional noise.
\newblock {\em Ann. Oper. Res.}, 189:255--276, 2011.

\bibitem{Kou2008}
S.~C. Kou.
\newblock Stochastic modeling in nanoscale biophysics:\ subdiffusion within
  proteins.
\newblock {\em Ann. Appl. Stat.}, 2(2):501--535, 2008.

\bibitem{KouSunneyXie2004}
S.~C. Kou and X.~{Sunney Xie}.
\newblock Generalized {L}angevin equation with fractional {G}aussian noise:\
  subdiffusion within a single protein molecule.
\newblock {\em Phys. Rev. Lett.}, 93(18), 2004.

\bibitem{Kubo1966}
R.~Kubo.
\newblock The fluctuation--dissipation theorem.
\newblock {\em Rep. Progr. Phys.}, 29(1):255--284, 1966.

\bibitem{LiLiu2019}
L.~Li and J.-G. Liu.
\newblock A discretization of {C}aputo derivatives with application to time
  fractional {SDE}s and gradient flows.
\newblock {\em SIAM J. Numer. Anal.}, 57(5):2095--2120, 2019.

\bibitem{LiLiu2017}
L.~Li, J.-G. Liu, and J.~Lu.
\newblock Fractional stochastic differential equations satisfying
  fluctuation-dissipation theorem.
\newblock {\em J. Stat. Phys.}, 169(2):316--339, 2017.

\bibitem{McKinleyNguyen2018}
S.~A. McKinley and H.~D. Nguyen.
\newblock Anomalous diffusion and the generalized {L}angevin equation.
\newblock {\em SIAM J. Math. Anal.}, 50(5):5119--5160, 2018.

\bibitem{Mori1965}
H.~Mori.
\newblock Transport, collective motion, and {B}rownian motion.
\newblock {\em Progr. Theoret. Phys.}, 33(3):423--455, 1965.

\bibitem{Nualart2006}
D.~Nualart.
\newblock {\em The {M}alliavin Calculus and Related Topics}.
\newblock Springer-Verlag, Berlin, second edition, 2006.

\bibitem{RichardTan2021}
A.~Richard, X.~Tan, and F.~Yang.
\newblock Discrete-time simulation of stochastic {V}olterra equations.
\newblock {\em Stochastic Process. Appl.}, 141:109--138, 2021.

\bibitem{Sanz-Sole2005}
M.~Sanz-Sol\'{e}.
\newblock {\em Malliavin Calculus with Applications to Stochastic Partial
  Differential Equations}.
\newblock EPFL Press, distributed by CRC Press, 2005.

\bibitem{ZhengWang2020}
X.~Zheng and H.~Wang.
\newblock An error estimate of a numerical approximation to a hidden-memory
  variable-order space-time fractional diffusion equation.
\newblock {\em SIAM J. Numer. Anal.}, 58(5):2492--2514, 2020.

\end{thebibliography}

\end{document}